\documentclass[reqno,12pt]{amsart}
\usepackage{latexsym}
\usepackage{amsmath,amssymb}
\usepackage{amsthm}
\usepackage{amsfonts}
\newtheorem{Definition}{Definition}[part]

\newtheorem{Lemma}{Lemma}[part]
\newtheorem{Proposition}{Proposition}[part]
\newtheorem{Remark}{Remark}[part]
\newtheorem{Theorem}{Theorem}[part]

\numberwithin{Assumption}{section} \numberwithin{Corollary}{section}
\numberwithin{Definition}{section} \numberwithin{equation}{section}
\numberwithin{Example}{section} \numberwithin{Lemma}{section}
\numberwithin{Proposition}{section} \numberwithin{Remark}{section}
\numberwithin{Theorem}{section}

\def\x{\xi}

\date{}
\title{Global weak solutions to 3D compressible Primitive equations with density-dependent viscosity}
\author [Fengchao Wang~~and~~ ]{}

\date{}

\begin{document}
\maketitle
\centerline{\scshape FENGCHAO WANG }
\medskip
{\footnotesize
  \centerline{School of Mathematical Sciences, Capital Normal University,
Beijing  100048, P. R. China}
   \centerline{  \it Email:wfcwymm@163.com}
   }
\vspace{3mm}
  \centerline{\scshape CHANGSHENG DOU }
  \medskip
{\footnotesize
  \centerline{School of Statistics, Capital University of Economics and Business,
 Beijing 100070, P.R. China}
   \centerline{  \it Email:  douchangsheng@cueb.edu.cn}
   }
\vspace{3mm}
\centerline{\scshape QUANSEN JIU }
  \medskip
{\footnotesize
  \centerline{School of Mathematical Sciences, Capital Normal University,
 Beijing  100048, P. R. China}
   \centerline{  \it Email: jiuqs@cnu.edu.cn}
   }

\pagestyle{myheadings} \thispagestyle{plain}\markboth{} {\small{global weak solutions to 3D Compressible Primitive equations }}

\vspace{3mm}

\textbf{Abstract:}
This paper is devoted to investigating the global existence of weak solutions for the compressible primitive equations (CPE) with damping term in a three-dimensional torus for large initial data. The system takes into account density-dependent viscosity. In our proof, we  represent the vertical velocity  as a function of the density and the  horizontal velocity which will play a role  to use the Faedo-Galerkin method to obtain the global existence of the approximate solutions. Motivated by Vasseur and Yu \cite{yucheng2016},  we obtain the key estimates of lower bound of the density, the Bresch-Desjardin entropy  on the approximate solutions. Based on these estimates, using compactness arguments, we prove the global existence of weak solutions of CPE by vanishing the parameters in our approximate system step by step.


\section{Introduction}
 \setcounter{equation}{0}
\setcounter{Assumption}{0} \setcounter{Theorem}{0}
\setcounter{Proposition}{0} \setcounter{Corollary}{0}
\setcounter{Lemma}{0}

In this paper, we consider the following compressible
primitive equations(CPE) with density-dependent viscosity
coefficients:
\begin{equation}\label{a1}
\left\{\begin{array}{lll}
\partial_{t}\rho+{\rm div}_{x}(\rho u)+\partial_{y}(\rho v)=0,\\
\partial_{t}(\rho u)+{\rm div}_{x}(\rho u\otimes u)+\partial_{y}(\rho uv)+\nabla_{x}P(\rho)+r\rho|u|u\\
=2{\rm div}_{x}(\nu_{1}D_{x}(u))+\partial_{y}(\nu_{2}\partial_{y}u),\\
\partial_{y}P(\rho)=-g\rho,P(\rho)=c^{2}\rho.
\end{array}\right.
\end{equation}
Here $(t,x,y)$ are independent variables, $t$ is the time, $x$ and $y$ are the horizontal and vertical coordinate with $x=(x_{1},x_{2})$, ${\rm div}_{x}=\partial_{x_{1}}+\partial_{x_{2}}$, $D_{x}(u)$ is the strain tensor with $D_{x}(u)=\frac{\nabla u+\nabla u^{\mathrm{T}}}{2}$, $\rho$ is the density of the medium, $\textbf{u}=(u,v)$ is the velocity, where $u$ is the horizontal component
and $v$ is the vertical component, $P$ is the pressure, $g>0$ is the free fall acceleration, and $(\nu_{1},\nu_{2})$ are the turbulence viscosities in the horizontal and vertical directions respectively.

Suppose that motion of the medium occurs in a bounded domain $\widetilde{\Omega}=\{(x,y);x\in\widetilde{\Omega}_{x},0<y<H\}$ where $\widetilde{\Omega}_{x}=T^{2}$ is the bi-dimensional torus and the boundary conditions on $\partial\widetilde{\Omega}$ are expressed by the relations
  \begin{equation}\label{a2}
  \begin{aligned}
  & periodic\ condtions\ on \ \partial\widetilde{\Omega}_{x},\\
  &v_{|y=0}=v_{|y=H}=0,\\
  &\partial_{y}u|_{y=0}=\partial_{y}u|_{y=H}=0.
  \end{aligned}
  \end{equation}
  The distribution of the horizontal component of the velocity and the density distribution are known at the initial time $t=0$:
\begin{equation}\label{a3}
 \rho u(0,x,y)=\widetilde{m}_{0}, \ \rho(0,x,y)=\rho_{0}.
\end{equation}

The primitive equations (PEs) of the atmosphere model are fundamental one of geophysical fluid dynamics (see \cite{a1987,BDL2003}). In the hierarchy of geophysical fluid dynamics models, they are situated between non-hydrostatic models and shallow water models.
The PEs are based on the so-called hydrostatic approximation, in which the conservation of momentum in the vertical direction is replaced by the hydrostatic equation. It was Richaedson \cite{rson1965} that first introduced the PEs in 1922, which consisted of the hydrodynamic thermodynamic equations with Coriolis force. Later in $1950$, Neumman and Charney \cite{nr1950} offered some simpler models such as quasi-geostrophic (QG) equations.

The PEs have attracted many attentions in fluid mechanics and applied mathematics such as \cite{gbh2011,ptw2004,ggm2001,rart2008,cske2007,zm2004,nt2011}, due to its physical importance,complexity, rich phenomena and mathematical challenges. Mathematical arguments of incompressible primitive equations were studied from 1990s. Lions, Temam and Wang first made the mathematical analysis of the PEs for atmosphere models in \cite{lions1992}. These authors have taken into account evaporation and solar heating with constant viscosities and they produced the mathematical formulation in two-and three- dimensions based on the works of J.Leray and obtained the global existence of weak solutions.
Bresch et al \cite{bd2003} got the existence and uniqueness of weak solutions for the two-dimensional hydrostatic Navier-Stokes equations. However, the uniqueness of weak solutions in three-dimensional case is still unclear.
The existence of strong solutions of the PEs have been widely studied. In \cite{zm2004} Temam and Ziane considered  the local existence of strong solutions for the PEs of the atmosphere, the ocean and the coupled atmosphere-ocean. Petcu et al. \cite{ptw2004} considered some regularity results for the two dimensional PEs with periodical boundary conditions. Cao and Titi \cite{cske2007} proved the global existence and uniqueness of
strong solutions to the three-dimensional viscous incompressible primitive equations of large scale ocean
and atmosphere dynamics. In the recent years, they also established the global existence of strong solutions for the system with only vertical diffusion or horizontal eddy diffusivity. Moreover, finite-time blowup of the solution for the inviscid primitive equation, with or without coupling to the heat equation has been shown by Cao et al.\cite{cske2015}.


In order to understand the mechanism of long-term weather prediction and climate, some mathematicians begin to study the compressible primitive equation of the atmosphere. As far as we know, the mathematical analysis to CPE is much more difficult. In the case that the viscosity coefficients are constant, Gatapov and Kazhikhov obtained a global existence of weak solutions for a two dimensional CPE by taking the classical Schauder fixed point theorem in \cite{bv2005}.
Ersoy and Ngom \cite{nt2012} showed a similar result of \cite{bv2005} for 2D compressible primitive equations with the horizontal viscosity $\nu_{1}(t,x,y)=\nu_{0} e^{-g/c^{2}y}$ for some given positive constant $\nu_{0}$ and the vertical viscosity $\nu_{2}$ any given function. From the point of view of physics, the viscosity may depend on density.
 In \cite{nt2011}   Ersoy,  Ngom and Sy derived a 3D CPE with a friction term
which is deduced from the 3D compressible and anisotropic Navier-Stokes equations by hydrostatic approximation and obtained the stability of the weak solutions. In \cite{ttg2015}, Tang and Gao studied the stability of weak solutions for two dimensional CPE model by used a new entropy estimate.
 Since the viscosity coefficient is
density-dependent which degenerates at vacuum, it is not available to obtain the estimate of the velocity itself. How to construct approximate solutions is a major and difficult issue in this case. Up to now, the global existence of weak solution for 3D CPE remains open as well.

 Recently, the global existence of weak solutions for the barotropic
compressible  Navier-Stokes equations with degenerate viscosity was proved by Vasseur and Yu  \cite{yucheng2016} and Li and Xin \cite{lx15}. Motivated by the work of Vasseur and Yu \cite{yucheng2016} and Feireisl \cite{feireisl2004},  we will investigate the global existence of weak solutions to CPE  (\ref{a1}) in this paper.  In our proof, we will first face a difficulty of how to  estimate the vertical velocity $w$ since there is no equations on it.   In order to overpass this  difficulty, we represent the vertical velocity $w$ as a  function of the density $\xi$ and the horizontal velocity $u$ and  use  Faedo-Galerkin method to prove the existence of the approximate system. Then, as mentioned above, the key issue in our proof is to construct the approximate solutions satisfying lower bound of the density and Bresch-Desjardin entropy (see \cite{BD1, dbcl2003}). Similar to  \cite{yucheng2016}, we construct the approximate solutions by adding viscosity term in the continuity equation, adding drag, cold pressure, quantum and higher derivative terms in the  momentum equation (see \eqref{b1} for details).  It should be noted that the damping term in \eqref{a1} can be dropped by proving a Mellet-Vasseur type inequality (see \cite{MV, yucheng2016,lx15}).

The rest of the paper is organized as follows. In section 2, we present some elementary inequalities and compactness theorems which will be used frequently in the whole proof. In section 3, we show the existence of global weak solutions to the approximate system by using the Faedo-Galerkin method. In section 4, we deduce the Bresch-Dejardins entropy estimates and pass to the limits as $\varepsilon, \mu\rightarrow0$. In section 5-6, by using the standard compactness arguments, we pass to the limits as $\eta\rightarrow0$ and $\kappa=\delta\rightarrow0$ step by step.
\section{Preliminaries and main result}

 In this section, we first introduce some basic inequalities needed later and then state our main result.
The following is the Gagliardo-Nirenberg inequality.
\begin{Lemma}\label{le1}(see\cite{lng1959})(Gagliardo-Nirenberg interpolation inequality)
For a function $u:\Omega\rightarrow \mathbb{R}$ defined on a bounded Lipschitz domain $\Omega\subset \mathbb{R}^{n}$,  $\forall\ 1 \leq q, r\leq \infty$ and a natural number $m$, Suppose also that a real number $\theta$ and a natural number $j $ are such that
$$\frac{1}{p}=\frac{j}{n}+(\frac{1}{r}-\frac{m}{n})\theta+\frac{1-\theta}{q}$$
and
$$\frac{j}{m}\leq \theta\leq 1,$$
then we have
$$\|D^{j}u\|_{L^{p}}\leq C_{1}\|D^{m}u\|_{L^{r}}^{\theta}\|u\|_{L^{q}}^{1-\theta}
+C_{2}\|u\|_{L^{s}}$$
where $s>0$ is arbitrary naturally, the constants $C_{1}$ and $C_{2}$ depend upon the domain $\Omega$ as well as $m,n,j,r,q,\theta$.
\end{Lemma}
The following two compactness results are known:
\begin{Lemma}\label{le2} (see \cite{jsn1987})
Let $X_{0},X$ and $X_{1}$ be three Banach spaces with $X_{0}\subseteq X \subseteq X_{1}$. Suppose that $X_{0}$ is compactly embedded in $X$ and  $X$ is continuously embedded in $X_{1}$. For $1\leq p,q\leq +\infty$, let
$$W=\{u\in L^{p}([0,T];X_{0}) \mid \partial_{t} u\in L^{q}([0,T];X_{1})\}.$$
(I) If $p<+\infty$, then the embedding of $W$ into $L^{p}([0,T];X)$ is compact.\\
(II) If $p=+\infty$ and $q>1$, then the embedding of $W$ into $C([0,T];X)$ is compact.
\end{Lemma}
\begin{Lemma}\label{le3} Suppose that $\{f_{n}\}\subset L^{p}(\Omega)$ and $\|f_{n}\|_{L^{p}}\leq C,C\ is\  constant.$ $f_{n}\rightarrow f$a.e.in $\Omega$, a bounded measurable set in $\mathbb{R}^{n}$,Then

(I) $f\in L^{p}(\Omega)$\ and\ $\|f\|_{L^{p}}\leq C$.\\
(II) $f_{n}\rightarrow f \ \ strongly\ \ in\ L^{\overline{p}}$,\ \ for\ any $\overline{p}\in [1,p)$.
\begin{proof}
Thanks to Fatou's lemma, we can prove $(I)$.

Next,
since $f_{n}\rightarrow f$ a.e. in $\Omega$, $f_{n}$ is uniformly bounded in $L^{p} (\Omega)$, and\ $mes\Omega\ <\infty.$  Due to the Lebesgue theorem, and applying H\"{o}lder inequality, we have
$$ \forall \ \epsilon>0,\ \ \Omega_{n}=\{x\in\Omega;\mid f_{n}(x) -f(x)\mid\geq \epsilon\}, \ mes\Omega_{n}\rightarrow  0 \ (\ n\rightarrow\infty)$$
and
\begin{equation*}
\begin{aligned}
\int_{\Omega}|f_{n}-f|^{\overline{p}}dx&=\int_{\Omega_{n}}
|f_{n}-f|^{\overline{p}}dx+\int_{\Omega-\Omega_{n}}|f_{n}-f|^{\overline{p}}dx\\
&\leq (mes\Omega_{n})^{1-\frac{\overline{p}}{p}}\parallel f_{n}- f\parallel_{p}^{\overline{p}}+\epsilon^{\overline{p}}mes\Omega
\rightarrow 0.
\end{aligned}
\end{equation*}
\end{proof}
\end{Lemma}

Before we state our main result, we  reformulate the system \eqref{a1} as follows. For simplicity, we assume $g=c^{2}=1$ in \eqref{a1}. From the hydrostatics equation
$$\rho(t,x,y)=\xi(t,x)e^{-y}$$
where $\xi(t,x)$ is an  unknown function. Following Ersoy and Ngom \cite{nt2012}, we set
$$
z=1-e^{-y},\ \ \ w(t,x,y)=e^{-y} v(t,x,y)
$$
and assume
$$
\nu_{1}(t,x,y)=\overline{\nu}_{1}\rho(t,x,y),\ \ \ \nu_{2}(t,x,y)=\overline{\nu}_{2}\rho(t,x,y)e^{2y}\ \ with\ \ \overline{\nu}_{i}>0.
$$
Then, the equations (\ref{a1}) reduce to
\begin{equation}\label{a4}
\left\{\begin{array}{lll}
\partial_{t}\xi+{\rm div}_{x}(\xi u)+\partial_{z}(\xi w)=0,\\
\partial_{t}(\xi u)+{\rm div}_{x}(\xi u\otimes u)+\partial_{z}(\xi uw)+\nabla_{x}\xi+r\xi |u|u\\
=2\overline{\nu}_{1}{\rm div}_{x}(\xi D_{x}(u))+\overline{\nu}_{2}\partial_{z}(\xi \partial_{z}u),\\
\partial_{z}\xi=0.
\end{array}\right.
\end{equation}
Here $t>0,(x,z)\in \Omega=\{x\in\Omega_{x},0<z<1-e^{-H}\}$, where $\Omega_{x}=T^{2}$ is the bi-dimensional torus.

Correspondingly, the boundary conditions have the form
\begin{equation}\label{a5}
  \begin{aligned}
  & periodic\ condtions\ on \ \partial\Omega_{x},\\
  &w_{|z=0}=w_{|z=h}=0,\\
  &\partial_{z}u|_{z=0}=\partial_{z}u|_{z=h}=0.
  \end{aligned}
  \end{equation}
and the initial data can be imposed as
\begin{equation}\label{a6}
 \xi u(0,x,y)=m_{0}(x,z), \ \
 \xi(0,x)=\xi_{0}(x).
\end{equation}
where $h=1-e^{-H}$ ,\ $ m_{0}(x,z)=\widetilde{m}_{0}(x,z)/(1-z)$ and $\xi_{0}(x)$ is a bounded non-negative function
$$
0\leq\xi_{0}(x)\leq M <+\infty.
$$

We consider system (\ref{a4}) with the following initial conditions
 \begin{equation}\label{a7}
  \begin{aligned}
  &\xi_{0}\in L^{1}(\Omega),\ \ \xi_{0}\ln\xi_{0}-\xi_{0}+1\in L^{1}(\Omega),\ \  \xi_{0}\geq 0,\ \ \nabla_{x}\sqrt{\xi_{0}}\in L^{2}(\Omega), \\
  & m_{0}=0\ \ if\ \ \xi_{0}=0,\ \ \ \frac{|m_{0}|^{2}}{\xi_{0}}\in L^{1}(\Omega).
  \end{aligned}
  \end{equation}
The defininition of the weak solutions is presented as follows.
\begin{Definition}\label{def1}

We call $(\xi, u,w)$ a weak solution to  the problem (\ref{a4})-(\ref{a6}) provided that
\begin{enumerate}
  \item $\xi$, $u$, $w$ belong to the classes
  \begin{equation}\label{a8}
  \left\{\begin{array}{lll}
  \xi\in L^{\infty}((0,T);L^{3}(\Omega)), \ \sqrt{\xi}u\in L^{\infty}((0,T);L^{2}(\Omega)),\\
  \sqrt{\xi}\in L^{\infty}((0,T);H^{1}(\Omega)),\ \ \x^{\frac{1}{3}}u\in L^{2}((0,T);L^{2}(\Omega)),\\
  \sqrt{\xi}\nabla_{x} u\in L^{2}((0,T);L^{2}(\Omega)),\ \ \sqrt{\xi}w\in L^{2}((0,T);L^{2}(\Omega)),\\ \sqrt{\xi}\partial_{z}w\in L^{2}((0,T);L^{2}(\Omega)),\ \ \sqrt{\xi}\partial_{z}u\in L^{2}((0,T);L^{2}(\Omega))
  \end{array}\right.
  \end{equation}
  \item
  The equations (\ref{a4}) hold in the sense of $\mathcal{D}^{'}((0,T)\times\Omega)$,
  \item
  (\ref{a6}) holds in $\mathcal{D}^{'}(\Omega)$.
\end{enumerate}
\end{Definition}
Formally, multiplying the momentum equation $(\ref{a4})_{2}$ by $u$ and integrating by parts we can deduce the following energy inequality
\begin{equation}\label{a9}
\begin{split}
&\frac{d}{dt}\int_{\Omega}(\xi\frac{u^{2}}{2}+(\xi \ln\xi-\xi+1))dxdz+\int_{\Omega}\xi(2\overline{\nu}_{1}|D_{x}(u)|^{2}+\overline{\nu}_{2}|\partial_{z}u|^{2})dxdz\\
&+r\int_{\Omega}\xi|u|^{3}dxdz\leq0.
\end{split}
\end{equation}
To deal with the unknown vertical velocity  $w$, we represent it as a function of a function of the density and the  horizontal velocity. Differentiating $(\ref{a4})_{1}$ with respect to $z$, we obtain
 \begin{equation}\label{a10}
\begin{split}
 &\partial_{zz}w=-\frac{\partial_{z}{\rm div}_{x}(\xi u)}{\xi},\\
 &w\mid_{z=0}=w\mid_{z=h}=0.
\end{split}
\end{equation}
Solving \eqref{a10} yields
\begin{equation}\label{a11}
w(z)=-\frac{{\rm div}_{x}(\xi\widetilde{u}(z))}{\xi}+z\frac{{\rm div}_{x}(\xi\overline{u})}{\xi}.
\end{equation}
Then
\begin{equation}\label{a12}
\partial_{z}w(z)=-\frac{{\rm div}_{x}(\xi u(z))}{\xi}+\frac{{\rm div}_{x}(\xi\overline{u})}{\xi},
\end{equation}
where
\begin{equation*}
\widetilde{u}(z)=\int^{z}_{0}u(\tau)d\tau,\ \ \ \ \ \ \ \overline{u}=\frac{1}{h}\int^{h}_{0}u(\tau)d\tau.
\end{equation*}
Substitute (\ref{a11})-(\ref{a12}) into  (\ref{a4}) to arrive
\begin{equation}\label{a13}
\left\{\begin{array}{lll}
\partial_{t}\xi+{\rm div}_{x}(\xi \overline{u})=0,\\
\partial_{t}(\xi u)+{\rm div}_{x}(\xi u\otimes u)+\partial_{z}(\xi u(-\dfrac{{\rm div}_{x}(\xi\widetilde{u}(z))}{\xi}+z\dfrac{{\rm div}_{x}(\xi\overline{u})}{\xi}))+\nabla_{x}\xi+r\xi |u|u\\
=2\overline{\nu}_{1}{\rm div}_{x}(\xi D_{x}(u))+\overline{\nu}_{2}\partial_{z}(\xi \partial_{z}u),\\
\partial_{z}\xi=0.
\end{array}\right.
\end{equation}

However, the above energy estimate (\ref{a9}) is not enough to prove the stability of the weak solutions $(\xi,u,w)$ of (\ref{a4}), we will obtain the following Bresch-Desjardins  entropy:
\begin{equation}\label{a14}
  \begin{split}
  &\int_{\Omega}\frac{1}{2}\xi (u+2\overline{\nu}_{1}\frac{\nabla_{x} \xi}{\xi})^{2}+(\xi \ln\xi-\xi+1) \ dxdz+
  \int^{T}_{0}\int_{\Omega}8\overline{\nu}_{1}|\nabla_{x}\sqrt{\xi}|^{2}dxdzdt\\
  &+\int^{T}_{0}\int_{\Omega}2\overline{\nu}_{1}\xi|\partial_{z}w|^{2}
  +\overline{\nu}_{2}\xi|\partial_{z}u|^{2}
  +\overline{\nu}_{1}\xi|A_{x}(u)|^{2}dxdzdt
  +r\int_{0}^{T}\int_{\Omega}| u|^{3} dxdt \\
  &\leq\int_{\Omega}\frac{1}{2}\xi_{0} (u_{0}+2\overline{\nu}_{1}\frac{\nabla_{x} \xi_{0}}{\xi_{0}})^{2}+(\xi_{0} \ln\xi_{0}-\xi_{0}+1) \ dxdz +C,
  \end{split}
  \end{equation}
  where C is bounded by the initial energy.

Now we state our result as follows.
\begin{Theorem}\label{theorem1}
Suppose that the initial data satisfies (\ref{a7}). Then, for any  $T>0$, (\ref{a4})-(\ref{a6}) has a weak solution $(\xi, u,w)$ in the sense of Definition \ref {def1} satisfying (\ref{a9}) and (\ref{a14}).
\end{Theorem}
\begin{Remark}
  The expression $w(z)=-\frac{{\rm div}_{x}(\xi\widetilde{u}(z))}{\xi}+z\frac{{\rm div}_{x}(\xi\overline{u})}{\xi}$ is crucial to allow us to use the Faedo-Galerkin method in Section 3.
\end{Remark}
\begin{Remark}
 The damping term $\xi|u|u$ here is used to give the strong convergence of $\sqrt{\xi}u$ in $L^{2}(0,T;L^{2}(\Omega)$. Following the approach in \cite{yucheng803}, we can similarly deduce the Mellet-Vasseur inequality for weak solutions, and the damping term can be removed.
\end{Remark}
\begin{Remark}
If we assume that the viscous $\nu_{i}(t,x,y)=\overline{\nu}_{i}\rho(t,x,y),\ \ i=1,2$ we can also obtain the existence of weak solutions to the three-dimensional compressible
primitive equations. In this case,  $1-z$ is bounded  away from below with a positive constant.
\end{Remark}
\begin{Theorem}\label{theorem2}
Under similar argument to Theorem \ref{theorem1},   (\ref{a1})-(\ref{a3}) has a weak solution $(\xi, u,v)$ in the sense of Definition \ref {def1} if we replace $(\xi,u,w)$ by $(\rho,u,v)$.
\end{Theorem}

\section{Faedo-Galerkin approximation}
In this section, we prove the existence of solutions to the approximate system of  the original problem (\ref{a4})-(\ref{a6}) by using the Faedo-Galerkin method, similar to [\cite{feireisl2004}, Chapter. 7] and \cite{aji2010,yucheng2016}.
\subsection{Approximate system}
In order to prove the global existence of weak solutions for the compressible primitive equations, we consider the following approximate system:
\begin{equation}\label{b1}
\left\{\begin{array}{lll}
\partial_{t}\xi+{\rm div}_{x}(\xi \overline{u})=\varepsilon\Delta_{x}\xi,\\
\partial_{t}(\xi u)+{\rm div}_{x}(\xi u\otimes u)+\partial_{z}(\xi uw)+\nabla_{x}\xi+r_{0} u+r\xi |u|u+\mu\Delta^{2}u\\
=2\overline{\nu}_{1}{\rm div}_{x}(\xi D_{x}(u))+\overline{\nu}_{2}\partial_{z}(\xi \partial_{z}u)+\varepsilon\nabla_{x}\cdot\nabla_{x}u+
\eta\nabla_{x}\xi^{-10}+\\ \ \ \ \kappa\xi\nabla_{x}(\dfrac{\Delta_{x}\sqrt{\xi}}{\sqrt{\xi}})+\delta\xi\nabla_{x}\Delta_{x}^{5}\xi\\
\partial_{z}\xi=0.
\end{array}\right.
\end{equation}
The extra terms $\eta\nabla\xi^{-10}$ and $\delta\xi\nabla_{x}\Delta_{x}^{5}\xi$ are  necessary to keep the density bounded, and bounded away from below with a positive constant for all the time. This enables us to take $\frac{\nabla\xi}{\xi}$ as a test function to derive the B-D entropy. Moreover, the term $r_{0}u$ is used to control the density near the vacuum, $\xi|u|u$ is used to make sure that $\sqrt{\xi}u$ is strong convergence in $L^{2}(0,T;L^{2}(\Omega))$ at the last approximation level.

 For $T>0$, we define a finite-dimensional space $X_{n}$=span$\{\psi_{1},...\psi_{n}\}$,\ $n\in\mathbb{N}$, where $\psi_{n}$ is the eigenfunctions of the Laplacian:
 \begin{equation*}
  \begin{aligned}
  &-\Delta\psi_{n}=\lambda_{n}\psi_{n}, \ \  on \ \Omega,\\
  & periodic\ condtions\ on \ \partial\Omega_{x},\\
  &\partial_{z}\psi_{n}|_{z=0}=\partial_{z}\psi_{n}|_{z=h}=0.
  \end{aligned}
  \end{equation*}
 Here $ \{\psi_{i}\}$ is an orthonormal basis of $L^{2}(\Omega)$ which is also an orthogonal basis of $H^{1}(\Omega)$. Let $(\xi_{0},u_{0})\in C^{\infty}(\Omega)$ be some initial data satisfying $\xi_{0}\geq \varsigma>0$ for some $\varsigma>0$, and let the velocity $u\in C([0,T];X_{n})$ satisfying
$$u(x,t)=\sum_{i=1}^{n}\lambda_{i}(t)\psi_{i}(x,z),\ \ (t,x,z)\in [0,T]\times\Omega$$
 Since $X_{n}$ is a finite-dimensional space, all the norms are equivalence on $X_{n}$. Therefore,  $u$ is bounded in $C([0,T];C^{k}(\Omega))$ for any $k\in \mathbb{N}$ and there exists a constant $C>0$ depending on $k$ such that
\begin{equation}\label{b2}
\|u\|_{C([0,T];C^{k}(\Omega))}\leq C\|u\|_{C([0,T];L^{2}(\Omega))}.
\end{equation}
Then the approximate of  the continuity equation is defined as follows:
\begin{equation}\label{b3}
\begin{aligned}
\left\{\begin{array}{lll}
\partial_{t}\xi+{\rm div}_{x}(\xi \overline{u})=\varepsilon \Delta_{x}\xi,\\
\xi_{0}\in C^{\infty}(\Omega),\ \ \xi_{0}\geq \varsigma >0,
\end{array}\right.
\end{aligned}
\end{equation}
In order to show the well-posedness of the parabolic problem (\ref{b3}), we introduce the following lemma:
\begin{Lemma}\label{le4}\cite {ala1995}
Let $\Omega\subset \mathbb{R}^{3}$ be a bounded domain of class $C^{2,\alpha},\ \alpha\in (0,1)$ and let $u\in C([0,T];X_{n})$ be a given vector field. If the initial data $\xi_{0}\geq \varsigma\geq0,\ \xi_{0}\in C^{2,\alpha}(\overline{\Omega})$, then there exists a unique classical solution $\xi=\xi_{u}$ to the  problem (\ref{b3}), more specifically,
\begin{equation}\label{b4}
\begin{aligned}
\xi\in C([0,T];C^{2,\alpha}(\overline{\Omega})),\ \
\partial_{t}\xi\in C([0,T];C^{0,\alpha}(\overline{\Omega})).
\end{aligned}
\end{equation}
\end{Lemma}
Moreover, it is easy to prove that the system $(\ref{b3})$ exists an unique classical solution $\xi\in C^{1}([0,T];C^{11}(\overline{\Omega}))$ by using the bootstrap method and Lemma 3.1£¬ due to the given function $u\in C([0,T];X_{n})$ .

Furthermore, if $0<\underline{\xi}\leq \xi\leq \overline{\xi}$ and ${\rm div}_{x}\overline{u}\in L^{1}([0,T];L^{\infty}(\Omega))$, the maximum principle gives $$\xi(x,t)\geq 0.$$

We define $L\xi=\partial_{t}\xi+{\rm div}_{x}(\xi \overline{u})-\varepsilon\triangle_{x}\xi$. By direct calculation, we can obtain
\begin{equation*}
\begin{aligned}
&L(\overline{\xi}e^{\int_{0}^{T}\|{\rm div}_{x}\overline{u}\|_{L^{\infty}}dt})=\overline{\xi}e^{\int_{0}^{T}
\|{\rm div}\overline{u}\|_{L^{\infty}}dt}(\|{\rm div}_{x}\overline{u}\|
_{L^{\infty}}+{\rm div}_{x}\overline{u})\geq0,\\
&L(\underline{\xi}e^{-\int_{0}^{T}\|{\rm div}_{x}\overline{u}\|_{L^{\infty}}dt})\leq 0,\ \ \ \ L\xi=0.
\end{aligned}
\end{equation*}
 So $\overline{\xi}e^{\int_{0}^{T}\|{\rm div}_{x}\overline{u}\|_{L^{\infty}}dt}$ and $\underline{\xi}e^{-\int_{0}^{T}\|{\rm div}_{x}\overline{u}\|_{L^{\infty}}dt}$ are super solutions and sub solutions to the equation (\ref{b3}) due to $0<\underline{\xi}\leq \xi\leq \overline{\xi}$ respectively. By using the comparison principle, we can obtain
\begin{equation}\label{b5}
0<\underline{\xi}e^{-\int_{0}^{T}\|{\rm div}_{x}\overline{u}\|_{L^{\infty}}dt}\leq \xi(x,t)\leq \overline{\xi}e^{\int_{0}^{T}\|{\rm div}_{x}\overline{u}\|_{L^{\infty}}dt},\ \ \forall x \in \Omega,\ \ t\geq0.
\end{equation}

Next we will show that the solution of the equation (\ref{b3}) depends on the velocity $u$ continuously. Let $\xi_{1},\ \xi_{2}$ are two solutions with the same initial data $(\ref{b3})_{2}$, i.e.
$$\partial_{t}\xi_{1} +{\rm div}_{x}(\xi_{1}\overline{u}_{1})=\varepsilon\Delta_{x}\xi_{1},\ \  \partial_{t}\xi_{2} +{\rm div}_{x}(\xi_{2}\overline{u}_{2})=\varepsilon\Delta_{x}\xi_{2}.$$
Subtracting the above two equations, multiplying the result equation with $-\Delta_{x}(\xi_{1}-\xi_{2})+(\xi_{1}-\xi_{2})$ and integrating by parts with respect to $x$ over $\Omega$, we have
\begin{equation*}
\begin{aligned}
\sup_{t\in [0,\tau]}\|\xi_{1}-\xi_{2}\|_{H^{1}}\leq \tau C( \xi_{0},\varepsilon,\|\overline{u}_{1}\|_{L^{1}((0,\tau);W^{1,\infty})},\|\overline{u}_{1}\|_{L^{1}((0,\tau);W^{1,\infty})})
\|u_{1}-u_{2}\|_{H^{1}}.
\end{aligned}
\end{equation*}
Moreover, for $u\in C([0,T];X_{n})$ a given vector field, by using the bootstrap method and compactness analysis, we can prove
\begin{equation}\label{b6}
\begin{aligned}
\|\xi_{1}&-\xi_{2}\|_{C([0,\tau];C^{11}(\overline{\Omega}))}\\
&\leq \tau C( \xi_{0},\varepsilon,\|\overline{u}_{1}\|_{L^{1}((0,\tau);X_{n})},\|\overline{u}_{1}\|_{L^{1}((0,\tau);X_{n})})
\|u_{1}-u_{2}\|_{C([0,\tau];X_{n})}.
\end{aligned}
\end{equation}

According to the mass equation, we construct the operator $\mathcal{S}:C([0,T];X_{n})\rightarrow C([0,T];C^{11}(\overline{\Omega}))$ by $ \mathcal{S}(u)=\xi$. Due to Lemma 3.1, (\ref{b5}) and (\ref{b6}), we have the following proposition.
\begin{Proposition}\label{pro1}
If $0<\underline{\xi}\leq \xi\leq \overline{\xi},\ \xi_{0}\in C^{\infty}(\overline{\Omega}),\ \overline{u}\in C([0,T];X_{n})$, then there exists an operator $\mathcal{S}:C([0,T];X_{n})\rightarrow C([0,T];C^{11}(\overline{\Omega}))$ satisfying
\begin{itemize}
  \item $\xi=\mathcal{S}(u)$ is an unique solution to the problem (\ref{b3}).
  \item $0<\underline{\xi}e^{-\int_{0}^{T}\|{\rm div}_{x}\overline{u}\|_{L^{\infty}}dt}\leq \xi(x,t)\leq \overline{\xi}e^{\int_{0}^{T}\|{\rm div}_{x}\overline{u}\|_{L^{\infty}}dt},\ \ \forall x\in \Omega,\ \ t\geq0.$
  \item $
\|\mathcal{S}(u_{1})-\mathcal{S}(u_{2})\|_{C([0,\tau];C^{11}(\overline{\Omega}))}$\\
$\leq \tau C( \xi_{0},\varepsilon,\|u_{1}\|_{L^{1}((0,\tau);X_{n})},\|u_{1}\|_{L^{1}((0,T\tau);X_{n})})
\|u_{1}-u_{2}\|_{C([0,\tau];X_{n})}$,\\
 for any $\tau \in [0,T]$ and $u_{1},u_{2}\in M_{k}=\{u\in C([0,T];X_{n});\|u\|_{C([0,T];X_{n})}\leq k,\ t\in\ [0,T]\}.$
\end{itemize}
\end{Proposition}
\begin{Remark}\label{re1}
Proposition \ref{pro1} implies the operator $\mathcal{S}$ is Lipschitz continuous for sufficient small time $t$.
\end{Remark}
\subsection{Faedo-Galerkin approximation}
We intend to solve the momentum equation on the space $X_{n}$. For any test function $\varphi\in X_{n},$ the approximate solutions $u_{n}\in(0,T;X_{n})$ looking for are required to satisfy
\begin{equation}\label{b7}
\begin{aligned}
&\int_{\Omega}\xi_{n} u_{n}(T)\varphi dxdz-\int_{\Omega} m_{0}\varphi dxdz+\mu\int_{0}^{T}\int_{\Omega} \Delta u_{n}\cdot\Delta \varphi dxdzdt\\
&-\int_{0}^{T}\int_{\Omega}\xi_{n} u_{n}w_{n}\partial_{z}\varphi dxdzdt
-\int_{0}^{T}\int_{\Omega} (\xi_{n} u_{n}\otimes u_{n})\cdot\nabla_{x} \varphi dxdzdt\\
&+2\overline{\nu}_{1}\int_{0}^{T}\int_{\Omega} \xi_{n} D_{x}(u_{n})\cdot\nabla_{x} \varphi dxdzdt-\int_{0}^{T}\int_{\Omega} \xi_{n}{\rm div}_{x} \varphi dxdzdt\\
&+\overline{\nu}_{2}\int_{0}^{T}\int_{\Omega}\xi_{n}\partial_{z}u_{n}\partial_{z}\varphi dxdzdt+\eta\int_{0}^{T}\int_{\Omega} \xi_{n}^{-10}{\rm div}_{x}\varphi dxdzdt\\
&+\varepsilon\int_{0}^{T}\int_{\Omega} \nabla_{x}\xi_{n}\cdot\nabla_{x} u_{n}\varphi dxdzdt
+r_{0}\int_{0}^{T}\int_{\Omega} u_{n}\varphi dxdzdt\\
&+r\int_{0}^{T}\int_{\Omega} \xi_{n} |u_{n}|u_{n} \varphi dxdt-\delta\int_{0}^{T}\int_{\Omega} \xi_{n} \nabla_{x}\Delta_{x}^{5} \xi_{n}\varphi dxdzdt \\
&=-2\kappa\int_{0}^{T}\int_{\Omega}\Delta_{x}\sqrt{\xi_{n}}\nabla_{x}\sqrt{\xi_{n}}\varphi dxdzdt-\kappa\int_{0}^{T}\int_{\Omega}\Delta_{x}\sqrt{\xi_{n}}\sqrt{\xi_{n}}{\rm div}_{x}\varphi dxdzdt,
\end{aligned}
\end{equation}
where
$$
w_{n}=-\frac{{\rm div}_{x}(\xi_{n}\widetilde{u}_{n}))}{\xi_{n}}+z\frac{{\rm div}_{x}(\xi_{n}\overline{u}_{n})}{\xi_{n}}
$$
and
 $$\xi_{n}=\mathcal{S}(u_{n}).$$
To solve (\ref{b7}), we follow the same arguments as in \cite{feireisl2004,fnp2001,aji2010}, and introduce the following family of operators
$$\mathcal{M}[\xi]:X_{n}\rightarrow X_{n}^{*},\ \ <\mathcal{M}[\xi] u,v>=\int_{\Omega}\xi u\cdot vdx,\ \ u,\ v\in X_{n}.$$
 Due to Lax-Milgram theorem, the above operators $M[\xi]$ are invertible provided  $\xi\in L^{1}(\Omega)$ with $\xi >\underline{\xi}>0$, and then we have
$$\|\mathcal{M}^{-1}[\xi]\|_{\mathcal{L}(X_{n}^{*},X_{n})}\leq \underline{\xi}^{-1},$$
where $\mathcal{L}(X_{n}^{*},X_{n})$ is the set of bounded liner mappings from $X_{n}^{*}$ to $X_{n}$.
 Moreover,
 \begin{equation*}
 \begin{aligned}
 \mathcal{M}^{-1}[\xi_{1}]-\mathcal{M}^{-1}[\xi_{2}]&=\mathcal{M}^{-1}[\xi_{2}]\mathcal{M}[\xi_{2}]
 \mathcal{M}^{-1}[\xi_{1}]-\mathcal{M}^{-1}[\xi_{2}]\mathcal{M}[\xi_{1}]\mathcal{M}^{-1}[\xi_{1}]\\
 &=\mathcal{M}^{-1}[\xi_{2}](\mathcal{M}[\xi_{2}]-\mathcal{M}[\xi_{1}])\mathcal{M}^{-1}[\xi_{1}],
 \end{aligned}
 \end{equation*}
 which imply that the operators
 $\mathcal{M}^{-1}[\rho]$ is Lipschitz continuous in the sense
  $$\|\mathcal{M}^{-1}[\xi_{1}]-\mathcal{M}^{-1}[\xi_{2}]\|_{\mathcal{L}(X_{n}^{*},X_{n})}\leq C(n,\underline{\xi})\|\xi_{1}-\xi_{2}\|_{L^{1}(\Omega)}$$
  for $\xi_{1},\ \xi_{2}\in L^{1}(\Omega)$ with $\xi_{1},\ \xi_{2}\geq \underline{\xi}>0$.
Then the integral equation  (\ref{b7}) can be rephrased as follows£»
\begin{equation}\label{b8}
\begin{aligned}
u_{n}(t)=\mathcal{M}^{-1}[(\mathcal{S}(u_{n})(t)]\big(\mathcal{M}[\xi_{0}](u_{0})
+\int_{0}^{T}\mathcal{N}
(\mathcal{S}(u_{n}),u_{n})(s)ds\big),
\end{aligned}
\end{equation}
where
\begin{equation*}
\begin{aligned}
\mathcal{N}
(\mathcal{S}(u_{n}),u_{n})(s)&=2\overline{\nu}_{1} {\rm div}_{x}(\xi D_{x}(u_{n}))+\overline{\nu}_{2}\partial_{z}(\xi\partial_{z}u_{n})-{\rm div}_{x}(\xi u_{n}\otimes u_{n})-\partial_{z}(\xi u_{n}w_{n})\\
&\ \ \ -\mu \Delta^{2}u_{n}-\varepsilon\nabla_{x} \xi\cdot\nabla_{x} u_{n}+\kappa\xi\nabla_{x}(\frac{\Delta_{x}\sqrt{\xi}}{\sqrt{\xi}})-\nabla_{x}\xi+\eta\nabla\xi^{-10}\\
&\ \ \ -r_{0}u_{n}-r\xi|u_{n}|u_{n}+\delta\xi\nabla_{x}\Delta_{x}^{5}\xi.
\end{aligned}
\end{equation*}

In view of the Lipschitz continuous estimates for $\mathcal{S}$ and $\mathcal{M}^{-1}$, the equation (\ref{b8}) can be solved on a short  time $[0,T^{'}]$, where $T^{'}\leq T$, by using the fixed-point theorem on the Banach space $C([0,T];X_{n})$. Thus there exists a unique local-in-time solution $(\xi_{n},u_{n},w_{n}))$ to (\ref{b3}) and (\ref{b8}).

Next we will extend this obtained local solution to be a global one. Differentiating $(\ref{b7})$ with respect to time $t$, taking $\varphi=u_{n}$ and integrating by parts with respect to $x$ over $\Omega$, we have the following energy estimate
\begin{equation}\label{b9}
\begin{aligned}
&\int_{\Omega}\frac{d}{dt}\xi_{n}(\frac{u_{n}^{2}}{2})dxdz
+\int_{\Omega}\nabla_{x}\xi_{n}\cdot u_{n}dxdz+r_{0}\int_{\Omega}u_{n}^{2}dxdz+r\int_{\Omega}\xi_{n}|u_{n}|^{3}dxdz
\\&+\mu\int_{\Omega}|\Delta u_{n}|^{2}dxdz
+2\overline{\nu}_{1}\int_{\Omega}\xi_{n}|D_{x}u_{n}|^{2}dxdz+\overline{\nu}_{2}\int_{\Omega}\xi_{n}|\partial_{z}u_{n}|^{2}dxdz-\\
&\eta\int_{\Omega}\nabla_{x}\xi_{n}^{-10} u_{n}dxdz+\kappa\int_{\Omega}\frac{\Delta_{x}\sqrt{\xi_{n}}}{\sqrt{\xi_{n}}}{\rm div}_{x}(\xi_{n} u_{n})dxdz+\delta\int_{\Omega}\Delta_{x}^{5}\xi_{n}{\rm div}_{x}(\xi_{n} u_{n})dxdz\\
&=0
\end{aligned}
\end{equation}

Furthermore, we estimate the terms of the left hand side one by one:
\begin{equation}\label{b10}
\begin{aligned}
\int_{\Omega} \nabla_{x} \xi_{n}\cdot u_{n}dxdz&=\int_{\Omega}\nabla_{x}\ln\xi_{n}\cdot \xi_{n} u_{n}dxdz\\
&=\int_{\Omega}\ln\xi_{n}(\partial_{t}\xi_{n}-\epsilon\Delta_{x}\xi_{n}+
\partial_{z}(\xi_{n} w_{n}))dxdz \\
&=\frac{d}{dt}\int_{\Omega}\xi_{n}\ln\xi_{n}-\xi_{n}+1dxdz+4\epsilon\int_{\Omega}|\nabla_{x}\sqrt{\xi_{n}}|^{2}dxdz
\end{aligned}
\end{equation}
where we used (\ref{a11}),$\  \partial_{z}\xi=0$ and integration by parts.
Next we will deal with the cold pressure and high order derivative of the density terms as follows
\begin{equation}\label{b11}
\begin{aligned}
-\eta\int_{\Omega} \nabla_{x}\xi_{n}^{-10}\cdot u_{n}dxdz&=-\frac{10}{11}\eta\int_{\Omega} \xi_{n} u_{n}\cdot\nabla_{x} \xi_{n}^{-11}dxdz\\
&=\frac{10}{11}\eta\int_{\Omega} \xi_{n}^{-11}[\varepsilon\Delta_{x}\xi_{n}-\partial_{t}\xi_{n}-\partial_{z}(\xi_{n} w_{n})]dxdz\\
&=\frac{1}{11}\eta\frac{d}{dt}\int_{\Omega}\xi_{n}^{-10}dxdz+\frac{2}{5}
\eta\varepsilon\int_{\Omega}|\nabla_{x}\xi_{n}^{-5}|^{2}dxdz,
\end{aligned}
\end{equation}
\begin{equation}\label{b12}
\begin{aligned}
\delta\int_{\Omega} {\rm div}_{x}(\xi_{n} u_{n})\Delta_{x}^{5}\xi_{n}&=\delta\int_{\Omega} [\varepsilon\Delta_{x}\xi_{n}-\partial_{t}\xi_{n}-\partial_{z}(\xi_{n} w_{n})]\Delta_{x}^{5}\xi_{n} dxdz\\
&=\frac{\delta}{2}\frac{d}{dt}\int_{\Omega}|\nabla_{x}
\Delta_{x}^{2}\xi_{n}|^{2}dxdz+\delta\varepsilon\int_{\Omega}|\Delta_{x}^{3}\xi_{n}|^{2}dxdz.
\end{aligned}
\end{equation}
Finally, we will estimate the quantum term
\begin{equation}\label{b13}
\begin{aligned}
\kappa\int_{\Omega}\frac{\Delta_{x}\sqrt{\xi_{n}}}{\sqrt{\xi_{n}}}{\rm div}_{x}(\xi_{n} u_{n})dxdz&=\kappa\int_{\Omega}\frac{\Delta_{x}\sqrt{\xi_{n}}}{\sqrt{\xi_{n}}} [\varepsilon\Delta_{x}\xi_{n}-\partial_{t}\xi_{n}-\partial_{z}(\xi_{n} w_{n}]dxdz\\
&=\kappa\frac{d}{dt}\int_{\Omega}|\nabla_{x}\sqrt{\xi_{n}}|^{2}dxdz+\frac{\kappa\varepsilon}{2}
\int_{\Omega}\xi_{n}|\nabla_{x}^{2}\ln\xi_{n}|^{2}dxdz,
\end{aligned}
\end{equation}
where we used
\begin{equation*}
\begin{aligned}
2\xi_{n}\nabla_{x}(\frac{\Delta_{x}\sqrt{\xi_{n}}}{\sqrt{\xi_{n}}})&=2\xi_{n}\nabla_{x}({\rm div}_{x}(\frac{\nabla_{x}\sqrt{\xi_{n}}}{\sqrt{\xi_{n}}})-\nabla_{x}\sqrt{\xi_{n}}\cdot\nabla_{x}
(\frac{1}{\sqrt{\xi_{n}}}))\\
&=\xi_{n}{\rm div}_{x}(\nabla_{x}^{2}\ln\xi_{n})+\frac{1}{2}\xi_{n}\nabla_{x}(\nabla_{x}\ln\xi_{n})^{2}\\
&={\rm div}_{x}(\xi_{n}\nabla_{x}^{2}\ln\xi_{n}).
\end{aligned}
\end{equation*}
Substituting (\ref{b10})-(\ref{b13}) into (\ref{b9}), we have the following energy balance
\begin{equation}\label{b14}
\begin{aligned}
&\frac{d}{dt}E(\xi_{n},u_{n})+
4\epsilon\int_{\Omega}|\nabla_{x}\sqrt{\xi_{n}}|^{2}dxdz+r_{0}\int_{\Omega}u_{n}^{2}dxdz+r\int_{\Omega}\xi_{n}|u_{n}|^{3}dxdz\\
&+\int_{\Omega} 2\overline{\nu}_{1}\xi_{n}|D_{x}(u_{n})|^{2}dxdz
+\int_{\Omega}\overline{\nu}_{2}\xi_{n}|\partial_{z}u_{n}|^{2}dxdz+\mu\int_{\Omega} |\Delta u_{n}|^{2}dxdz\\
&+\frac{2}{5}\eta\epsilon\int_{\Omega}|\nabla_{x}\xi_{n}^{-5}|^{2}dxdz+\frac{\kappa\epsilon}{2}\int_{\Omega}\xi_{n}|\nabla_{x}\ln\xi_{n}|^{2}dxdz+
\delta\epsilon\int_{\Omega}|\Delta^{3}_{x}\xi_{n}|^{2}dxdz=0,
\end{aligned}
\end{equation}
on $[0,T^{'}]$, where
$$
E(\xi_{n},u_{n})=\int_{\Omega}\frac{1}{2}\xi_{n} u_{n}^{2}+\xi_{n}\ln\xi_{n}-\xi_{n}+1+\frac{1}{11}\eta\xi_{n}^{-10}+\kappa|\nabla_{x}\sqrt{\xi_{n}}|^{2}+
\frac{\delta}{2}|\nabla_{x}\Delta_{x}^{2}\xi_{n}|^{2}dxdz
$$
and
$$
E_{0}(\xi_{n},u_{n})=\int_{\Omega}\frac{1}{2}\xi_{0} u_{0}^{2}+\xi_{0}\ln\xi_{0}-\xi_{0}+1+\frac{1}{11}\eta\xi_{0}^{-10}+\kappa|\nabla_{x}\sqrt{\xi_{0}}|^{2}+
\frac{\delta}{2}|\nabla_{x}\Delta_{x}^{2}\xi_{0}|^{2}dxdz.
$$

So the energy inequality (\ref{b14}) yields
\begin{equation}\label{b15}
\int_{0}^{T^{'}}\|\Delta u_{n}\|_{L^{2}}^{2}dt\leq E_{0}(\xi_{n},u_{n})<+\infty,
\end{equation}
Thanks to dim$X_{n}<\infty$ and (\ref{b5}), the density is bounded and bounded away from blow with a positive constant, which means there exists a constant $c>0$ such that
\begin{equation}\label{b16}
0<\frac{1}{c}\leq \xi_{n}\leq c,
\end{equation}
for all $t\in [0,T^{'})$. Furthermore, the basic energy equality also gives us
\begin{equation}\label{b17}
\sup_{t\in(0,T^{'})}\int_{\Omega}\xi_{n}u_{n}^{2}dxdz\leq C<\infty,
\end{equation}

which together with (\ref{b16}), implies
\begin{equation}\label{b18}
\sup_{t\in(0,T^{'})}\|u_{n}\|_{L^{\infty}}\leq C<\infty,
\end{equation}
where we have used the fact that all the norms are equivalence on $X_{n}$. Then we can repeat above argument many times, we can extend $T^{'}$ to $T$ and obtain $u_{n}\in C([0,T];X_{n})$. Thus there exists a global solution $(\xi_{n},u_{n},w_{n}))$ to (\ref{b3}), (\ref{b8}) for any time $T$.
By the energy inequality(\ref{b14}), we have
$$
\sup_{t\in(0,T)}\int_{\Omega}\sqrt{\xi_{n}}u_{n}^{2}dxdz\leq E_{0}(\xi_{n},u_{n}).
$$

Then we can prove
\begin{equation*}
\begin{aligned}
\int_{\Omega}\xi_{n}\overline{u}_{n}^{2}dxdz&=\int_{\Omega}\xi_{n}(\int_{0}^{h}\frac{1}{h}u_{n}(\tau)d\tau)^{2}dxdz\\
&\leq\int_{\Omega}\xi_{n}\frac{1}{h^{2}}(\int_{0}^{h}1^{2}d\tau)(\int_{0}^{h}u_{n}^{2}(\tau)d\tau)dxdz\\
&\leq \frac{1}{h}\int_{0}^{h}E_{0}(\xi_{n},u_{n})dz\leq E_{0}(\xi_{n},u_{n})
\end{aligned}
\end{equation*}
and
\begin{equation*}
\begin{aligned}
\int_{\Omega}\xi_{n}\widetilde{u}_{n}^{2}dxdz&=\int_{\Omega}\xi_{n}(\int_{0}^{z}u_{n}(\tau)d\tau)^{2}dxdz\\
&\leq\int_{\Omega}\xi_{n}(\int_{0}^{h}1^{2}d\tau)(\int_{0}^{h}u_{n}^{2}(\tau)d\tau)dxdz\\
&\leq h\int_{0}^{h}E_{0}(\xi_{n},u_{n})dz\leq h^{2}E_{0}(\xi_{n},u_{n}).
\end{aligned}
\end{equation*}
Furthermore, the basic energy also gives us
$$
\sup_{t\in(0,T)}\int_{\Omega}\delta|\nabla_{x}\Delta_{x}^{2}\xi_{n}|^{2}dxdz\leq E_{0}(\xi_{n},u_{n}),
$$
which together with (\ref{b16}) implies that $\xi(x,t)$ is a positive smooth function for all $(x,t)$.
Then the  energy inequality allows us to present the following lemma (see\cite{yucheng2016,aji2010}).
\begin{Lemma}\label{lemma1}
For any smooth positive function $\xi(x,t)$, the following uniform estimate  holds
\begin{equation}\label{b19}
\begin{aligned}
(\kappa\varepsilon)^{\frac{1}{2}}\|\sqrt{\xi_{n}}\|_{L^{2}(0,T;H^{2}(\Omega))}+(\kappa\varepsilon)^{\frac{1}{4}}\|\nabla_{x}\xi_{n}^{\frac{1}{4}}\|_{L^{4}(0,T;L^{4}(\Omega))}\leq C,
\end{aligned}
\end{equation}
for some constant $C>0$ which is independent of $N$
\end{Lemma}

To conclude this part, we have the following proposition on the approximate solutions $(\xi_{n},u_{n},w_{n})$:
\begin{Proposition}\label{pro2}
Let $(\xi_{n},u_{n},w_{n})$ be the solutions of (\ref{b3}) and (\ref{b8}) on $(0,T)\times \Omega$ constructed above, then the solutions must satisfy the energy inequality
\begin{equation}\label{b20}
\begin{aligned}
&E(\xi_{n},u_{n})+
4\epsilon\int_{0}^{T}\int_{\Omega}|\nabla_{x}\sqrt{\xi_{n}}|^{2}dxdzdt+r_{0}\int_{0}^{T}\int_{\Omega}u_{n}^{2}dxdzdt\\
&+r\int_{0}^{T}\int_{\Omega}\xi_{n}|u_{n}|^{3}dxdzdt
+\int_{0}^{T}\int_{\Omega} 2\overline{\nu}_{1}\xi_{n}|D_{x}(u_{n})|^{2}dxdzdt\\
&+\int_{0}^{T}\int_{\Omega}\overline{\nu}_{2}\xi_{n}|\partial_{z}u_{n}|^{2}dxdzdt+\mu\int_{0}^{T}\int_{\Omega} |\Delta_{x}u_{n}|^{2}dxdzdt\\
&+\frac{2}{5}\eta\epsilon\int_{0}^{T}\int_{\Omega}|\nabla_{x}\xi_{n}^{-5}|^{2}dxdzdt+\frac{\kappa\epsilon}{2}\int_{0}^{T}\int_{\Omega}\xi_{n}|\nabla_{x}\ln\xi_{n}|^{2}dxdzdt\\
&+\delta\epsilon\int_{0}^{T}\int_{\Omega}|\Delta^{3}_{x}\xi_{n}|^{2}dxdzdt
\leq E(\xi_{0},u_{0}),
\end{aligned}
\end{equation}
where
$$
E(\xi_{n},u_{n})=\int_{\Omega}\frac{1}{2}\xi_{n} u_{n}^{2}+\xi_{n}\ln\xi_{n}-\xi_{n}+1+\frac{1}{11}\eta\xi_{n}^{-10}+\kappa|\nabla_{x}\sqrt{\xi_{n}}|^{2}+
\frac{\delta}{2}|\nabla_{x}\Delta_{x}^{2}\xi_{n}|^{2}dxdz.
$$

 In particular, we have
\begin{equation}\label{b21}
\begin{aligned}
&\sqrt{\xi_{n}}u_{n}\in L^{\infty}(0,T;L^{2}),\xi_{n}\ln\xi_{n}-\xi_{n}+1\in L^{\infty}(0,T;L^{1}),\eta\xi_{n}^{-10}\in L^{\infty}(0,T;L^{1}),\\
&\sqrt{\overline{\nu}_{1}}\sqrt{\xi_{n}}D_{x}u_{n}\in L^{2}(0,T;L^{2}),\sqrt{\mu}\Delta u_{n}\in L^{2}(0,T;L^{2}),\sqrt{\kappa}\sqrt{\xi_{n}}\in L^{\infty}(0,T;H^{1}),\\
&\sqrt{\delta}\xi_{n}\in L^{\infty}(0,T;H^{5}),\sqrt{\varepsilon}\nabla_{x} \sqrt{\xi}_{n}\in L^{2}(0,T;L^{2}),\sqrt{\overline{\nu}_{2}}\sqrt{\xi_{n}}\partial_{z}u_{n}\in L^{2}(0,T;L^{2})\\
&\sqrt{\varepsilon\eta}\nabla\xi_{n}^{-5}\in L^{2}(0,T;L^{2}),\sqrt{\delta\varepsilon}\xi_{n}\in L^{2}(0,T;H^{6}),\sqrt{r_{0}}u_{n}\in L^{2}(0,T;L^{2}),\\
&\xi_{n}^{\frac{1}{3}}u_{n}\in L^{3}(0,T;L^{3}),\sqrt{\kappa\varepsilon}\sqrt{\xi_{n}}\in L^{2}(0,T;H^{2}),
 \sqrt[4]{\kappa\varepsilon}\nabla_{x}\xi_{n}^{\frac{1}{4}}\in L^{4}(0,T;L^{4}),\\
 &\sqrt{\xi_{n}}\overline{u}_{n}\in L^{\infty}(0,T;L^{2}),\sqrt{\xi_{n}}\widetilde{u}_{n}\in L^{\infty}(0,T;L^{2}).
\end{aligned}
\end{equation}
\end{Proposition}
\subsection{Passing to the limits as $n\rightarrow\infty$.}We fix $\varepsilon,\eta,\kappa,\mu,\delta,r_{0}>0$ and take first the limit as $n\rightarrow\infty$. Here we need the regularities of the solution, Lemma 2.2 and the following compactness results.
\subsubsection{Convergence of $\xi_{n}$.}
\begin{Lemma}\label{22}
The following estimates hold for any fixed positive constants $\varepsilon,\eta,\kappa,$
$\mu,\delta,r_{0}$:
\begin{equation}\label{b22}
\begin{aligned}
&\|\partial_{t}\sqrt{\xi_{n}}\|_{L^{\infty}(0,T;W^{-1,\frac{3}{2}})}+\|\sqrt{\xi_{n}}\|_{L^{2}(0,T;H^{2})}\leq K,\\
&\|\xi_{n}\|_{L^{\infty}(0,T;H^{5})}+\|\partial_{t}\xi_{n}\|_{L^{\infty}(0,T;W^{-1,\frac{3}{2}})}\leq K,\ \ \|\xi_{n}^{-10}\|_{L^{\frac{5}{3}}(0,T;L^{\frac{5}{3}})}\leq K,\\
\end{aligned}
\end{equation}
where $K$ is independent of $n$, depends on $\varepsilon,\eta,\delta$,$r_{0}$, initial data and $T$. Moreover, up to an extracted subsequence
\begin{equation}\label{b23}
\begin{aligned}
&\sqrt{\xi_{n}}\rightarrow\sqrt{\xi}\ \   strongly\ in\ L^{2}([0,T];H^{1}) \ \ and \ \ \sqrt{\xi_{n}}\rightarrow\sqrt{\xi}\ \ a.e.  ,\\
&\xi_{n}\rightarrow\xi\ \ strongly\ in\  C([0,T];H^{5}) \ \ and\ \ \xi_{n}\rightarrow\xi\ \ a.e. ,\\
&\xi_{n}^{-10}\rightarrow \xi^{-10}\ \ strongly\ in\  L^{1}(0,T;L^{1}) \ \ and \ \ \xi_{n}^{-10}\rightarrow \xi^{-10}\ \ a.e.
\end{aligned}
\end{equation}
\end{Lemma}
\begin{proof}
Since$ \sqrt{\xi_{n}}\in L^{\infty}(0,T;H^{1})$, we deduce that $ \sqrt{\xi_{n}}\in L^{\infty}(0,T;L^{6})$. Then

$$
\xi_{n} \overline{u}_{n}=\sqrt{\xi_{n}}\sqrt{\xi_{n}}\overline{u}_{n}\in L^{\infty}(0,T;L^{\frac{3}{2}}).
$$
By (\ref{b3}), we have
\begin{equation}\label{b24}
\begin{aligned}
\partial_{t}\xi_{n}&=\varepsilon\Delta_{x}\xi_{n}-{\rm div}_{x}(\xi_{n}\overline{ u}_{n})\\
&=\varepsilon\Delta_{x}\xi_{n}-{\rm div}_{x}(\sqrt{\xi_{n}}\sqrt{\xi_{n}}\overline{u}_{n}) \in L^{\infty}(0,T;W^{-1,\frac{3}{2}})
\end{aligned}
\end{equation}
By using (\ref{b24}), $\xi_{n}\in L^{\infty}(0,T;H^{5})\cap L^{2}(0,T;H^{6})$ and Lemma \ref{le2}, we can claim $\xi_{n}\in C([0,T];H^{5})$. So up to a subsequence, we have
$$\xi_{n}\rightarrow\xi\ \ \ strongly\ \ \ in\ C([0,T];H^{5})\  and \ \ \xi_{n}\rightarrow\xi\ \  a.e.$$

Next we claim that $\xi_{n}^{-10}$ is bounded in $L^{\frac{5}{3}}(0,T;L^{\frac{5}{3}})$.
 Thanks to $\nabla\xi_{n}^{-5}$ is bounded in ${L^{2}(0,T;L^{2})}$, and the Sobolev embedding theorem, we obtain $\xi_{n}$ is bounded in ${L^{1}(0,T;L^{3})}$.
 Note that
 $$
 \xi_{n}^{-10}\in L^{\infty}(0,T;L^{1}).
 $$
 Then we apply H$\ddot{o}$lder inequality to have
\begin{equation}\label{b25}
\begin{aligned}\|\xi_{n}^{-10}\|_{L^{\frac{5}{3}}(0,T;L^{\frac{5}{3}})}\leq \|\xi_{n}^{-10}\|_{L^{\infty}(0,T;L^{1})}^{\frac{2}{5}}\|\xi_{n}^{-10}\|_{L^{1}(0,T;L^{3})}
^{\frac{3}{5}}\leq K.
\end{aligned}
\end{equation}
 Meanwhile, by using the following Sobolve inequality (see\cite{dbb2006,ez2012})
$$
\|\xi_{n}^{-1}\|_{L^{\infty}} \leq C(1+\|\xi_{n}^{-1}\|_{L^{3}})^{3}(1+\|\xi_{n}\|_{H^{k+2}}^{2}),
$$
for $k\geq\frac{3}{2}$, and the estimates of density in (\ref{b21}) ,we get
\begin{equation}\label{b26}
\begin{aligned}
\|\xi_{n}^{-1}\|_{L^{\infty}}\leq C(\eta,\delta),\ \ \ a.e.\in(0,T)\times\Omega.
\end{aligned}
\end{equation}

Thus, we have $\xi_{n}^{-10}$converges almost everywhere to $ \xi^{-10}$. Thanks to (\ref{b25}) and lemma \ref{le3}, we have
$$\xi_{n}^{-10}\rightarrow \xi^{-10}\ \ \ strongly\ in\ L^{1}(0,T;L^{1}).$$
Using again the mass equation (\ref{b3}), we have
$$\partial_{t}\sqrt{\xi_{n}}=\frac{1}{2}\frac{1}{\sqrt{\xi_{n}}}(\varepsilon\Delta_{x}\xi_{n}-{\rm div}_{x}(\sqrt{\xi_{n}}\sqrt{\xi_{n}}\overline{u}_{n})),$$
which combines with (\ref{b26}) and $\partial_{t}\xi_{n}\in L^{\infty}(0,T;W^{-1,\frac{3}{2}})$, we have
$$
\partial_{t}\sqrt{\xi_{n}}\in L^{\infty}(0,T;W^{-1,\frac{3}{2}}).
$$
Notice that $\sqrt{\xi_{n}}\in L^{2}(0,T;H^{2})$, using  Lemma \ref{le2}, we get
$$
\sqrt{\xi_{n}}\rightarrow\sqrt{\xi}\ \ \ \ strongly\ in\ L^{2}([0,T];H^{1}) \ \ \ \and\ \ \sqrt{\xi_{n}}\rightarrow\sqrt{\xi}\ \ a.e..
$$
Then the proof of this lemma is completed.
\end{proof}
\subsubsection{ Convergence of momentum $\xi_{n}u_{n}$}
\begin{Lemma}\label{l2}
Up to an extracted subsequence, we have
$$\xi_{n}u_{n}\rightarrow\xi u\ \ strongly\ \ in \ \ L^{2}(0,T;L^{2}) \ \ \ and \ \ \xi_{n}u_{n}\rightarrow\xi u \ \ a.e..$$
\end{Lemma}
\begin{proof}
From the energy estimates, we know that $u_{n}$ is bounded in $L^{2}(0,T;L^{2})$. So up to a subsequence, we have
$$u_{n}\rightharpoonup u\ \ \ \ in\ \ L^{2}(0,T;L^{2}).$$
Recalling that $\xi_{n}\rightarrow \xi$ strongly in $C([0,T];H^{5})$, we have
$$\xi_{n}u_{n}\rightarrow \xi u\ \ \ \ stronly\ \ in\ \ L^{1}(0,T;L^{1}).$$
Moreover, since $\xi_{n}\in L^{\infty}(0,T;H^{5}), u_{n}\in L^{2}(0,T;H^{2}),\sqrt{\xi_{n}}\partial_{z}u_{n}\in L^{2}(0,T;L^{2})$, we can show $$\nabla(\xi_{n}u_{n})=\nabla_{x}\xi_{n}u_{n}+\xi_{n}\nabla_{x} u_{n}+\xi_{n}\partial_{z}u_{n}\in L^{2}(0,T;L^{2}).$$
This identity and $\xi_{n}u_{n}\in L^{2}(0,T;L^{2})$ give $\xi_{n}u_{n}\in L^{2}(0,T;H^{1})$.
Next we can prove
$$\partial _{t}(\xi_{n}u_{n})\in L^{2}(0,T;H^{-s})\ \ \ for\ some\ s>0.$$
In fact,

\begin{equation}\label{b27}
\begin{aligned}
\partial_{t}(\xi_{n} u_{n})=&-{\rm div}_{x}(\xi_{n} u_{n}\otimes u_{n})-\partial_{z}(\xi_{n} u_{n}w_{n})-\nabla_{x}\xi_{n}-r_{0} u_{n}-r\xi |u_{n}|u_{n}-\mu\Delta^{2}u_{n}\\
&+2\overline{\nu}_{1}{\rm div}_{x}(\xi_{n} D_{x}(u_{n}))+\overline{\nu}_{2}\partial_{z}(\xi_{n} \partial_{z}u_{n})+\varepsilon\nabla_{x}\xi_{n}\cdot\nabla_{x}u_{n}
+\eta\nabla_{x}\xi_{n}^{-10}\\
&+\kappa\xi_{n}\nabla_{x}(\dfrac{\Delta_{x}\sqrt{\xi_{n}}}{\sqrt{\xi_{n}}})
+\delta\xi_{n}\nabla_{x}\Delta_{x}^{5}\xi_{n}.\\
\end{aligned}
\end{equation}
where
\begin{equation}\label{b28}
\begin{aligned}
\partial_{z}(\xi_{n} u_{n}w_{n})&=\partial_{z}(\xi_{n} u_{n}(-\frac{{\rm div}_{x}(\xi_{n}\widetilde{u}_{n})}{\xi_{n}}+z\frac{{\rm div}_{x}(\xi_{n}\overline{u}_{n})}{\xi_{n}}))\\
&=\partial_{z}(-{\rm div}_{x}(\xi_{n}\widetilde{u}_{n}\otimes u_{n})+\xi_{n}\widetilde{u}_{n}\cdot\nabla_{x}u_{n}\\
&\ \ \ +z{\rm div}_{x}(\xi_{n}\overline{u}_{n}\otimes u_{n})-z\xi_{n}\overline{u}_{n}\cdot\nabla_{x}u_{n}).
\end{aligned}
\end{equation}
Based on the energy estimates (\ref{b21}), it is easy to check that $\partial_{t}(\xi_{n}u_{n})\in L^{2}(0,T;$ $H^{-5})$. Then, we can show
$$\xi_{n}u_{n} \rightarrow g \ \ strongly\ \ in\ L^{2}(0,T;L^{2}),\  for\  some\  function\  g\in L^{2}(0,T;L^{2}).$$
Moreover, since $\xi_{n}u_{n}\rightarrow \xi u \ strongly\ in\ L^{1}(0,T;L^{1}),$ we have $$\xi_{n}u_{n} \rightarrow \xi u\ \ strongly\ \ in\ L^{2}(0,T;L^{2}).$$
 Thus the proof of this lemma is completed.
\end{proof}

\subsubsection{ Convergence of $\sqrt{\xi_{n}}u_{n}$,\ Convergence of $\sqrt{\xi_{n}}\widetilde{u}_{n}$,\ Convergence of $\sqrt{\xi_{n}}\overline{u}_{n}$}
\begin{Lemma}\label{l9}
Up to an extracted subsequence, we have
\begin{equation*}
\begin{aligned}
&\sqrt{\xi_{n}}u_{n}\rightarrow\sqrt{\xi} u\   \ strongly\ \ in \ \ L^{2}(0,T;L^{2}),\\
&\sqrt{\xi_{n}}\widetilde{u}_{n}\rightarrow\sqrt{\xi} \widetilde{u}\ \  strongly\ \ in \ \ L^{2}(0,T;L^{2}),\\
&\sqrt{\xi_{n}}\overline{u}_{n}\rightarrow\sqrt{\xi} \overline{u}\  strongly\ \ in \ \ L^{2}(0,T;L^{2}).
\end{aligned}
\end{equation*}
\end{Lemma}
Remark:$\sqrt{\xi_{n}}u_{n}\rightarrow\sqrt{\xi} u,\ \sqrt{\xi_{n}}\widetilde{u}_{n}\rightarrow\sqrt{\xi}\widetilde{u}\ and\ \sqrt{\xi_{n}}\overline{u}_{n}\rightarrow\sqrt{\xi} \overline{u}  \ \ a.e.$.
\begin{proof}
Due to the above basic energy estimates, Fatou's Lemma gives
$$
\int_{0}^{T}\int_{\Omega}\xi |u|^{3}dxdzdt\leq\int_{0}^{T}\int_{\Omega}\liminf\xi_{n}|u_{n}|^{3}dxdzdt
\leq\liminf\int_{0}^{T}\int_{\Omega}\xi_{n}|u_{n}|^{3}dxdzdt,
$$
and hence $\xi |u|^{3}$ is bounded in $L^{1}(0,T;L^{1}(\Omega))$.

By Lemma 3.4, up to subsequence, we have
$\sqrt{\xi_{n}}\rightarrow\sqrt{\xi}$
and
$\xi_{n}u_{n}\rightarrow\xi u$ almost everywhere. Then, it's easy to show that
$$\sqrt{\xi_{n}}u_{n}\rightarrow \sqrt{\xi}u,\ \ a.e. \ in\ \{\xi_{n}(x,t)\neq 0\}.$$
And for almost every $(x,z,t)$ in $\{\xi_{n}(x,t)=0\}$, we have
$$\sqrt{\xi_{n}}u_{n}\textbf{l}_{|u_{n}|\leq M}\leq M\sqrt{\xi_{n}}=0=\sqrt{\xi}u\textbf{l}_{|u|\leq M}.$$

As a matter of fact, $\sqrt{\xi_{n}}u_{n}\textbf{l}_{|u_{n}|\leq M}$ converges to $\sqrt{\xi}u\textbf{l}_{|u|\leq M}$ almost everywhere for $(x,t)$. Meanwhile, $\sqrt{\xi_{n}}u_{n}\textbf{l}_{|u_{n}|\leq M}$ is uniformly bounded in $L^{\infty}(0,T;L^{6})$. Using Lemma \ref{le3}, we obtain
\begin{equation}\label{b29}
\sqrt{\xi_{n}}u_{n}\textbf{l}_{|u_{n}|\leq M}\rightarrow \sqrt{\xi}u\textbf{l}_{|u|\leq M}\ \ strongly\ \ in\ L^{2}(0,T;L^{2}).
\end{equation}
For any $M>0$, we have
\begin{equation}\label{b30}
\begin{aligned}
&\int_{0}^{T}\int_{\Omega}|\sqrt{\xi_{n}}u_{n}-\sqrt{\xi}u|^{2}dxdzdt\\
&\leq 2\int_{0}^{T}\int_{\Omega}|\sqrt{\xi_{n}}u_{n}\textbf{l}_{|u_{n}|\leq M}-\sqrt{\xi}u\textbf{l}_{|u|\leq M}|^{2}dxdzdt\\
&\ \ \ +4\int_{0}^{T}\int_{\Omega}|\sqrt{\xi_{n}}u_{n}\textbf{l}_{|u_{n}|\geq M}|^{2}dxdzdt+4\int_{0}^{T}\int_{\Omega}|\sqrt{\xi}u\textbf{l}_{|u|\geq M}|^{2}dxdzdt\\
&\leq2\int_{0}^{T}\int_{\Omega}|\sqrt{\xi_{n}}u_{n}\textbf{l}_{|u_{n}|\leq M}-\sqrt{\xi}u\textbf{l}_{|u|\leq M}|^{2}dxdzdt\\
&\ \ \ +\frac{4}{M}\int_{0}^{T}\int_{\Omega}\xi_{n}|u_{n}|^{3}dxdzdt+
\frac{4}{M}\int_{0}^{T}\int_{\Omega}\xi|u|^{3}dxdzdt\\
\end{aligned}
\end{equation}

Thanks to (\ref{b29}) and $\xi |u|^{3},\ \xi_{n} |u_{n}|^{3}$ are bounded in $L^{1}(0,T;L^{1}(\Omega))$, letting $M\rightarrow\infty$, we have
$$\sqrt{\xi_{n}}u_{n}\rightarrow\sqrt{\xi}u\ \ strongly\ in\ L^{2}(0,T;L^{2}).$$

Recalling that $\overline{u}=\frac{1}{h}\int_{0}^{h}u(\tau)d\tau$, similar to the above proof, we have
\begin{equation}\label{b31}
\sqrt{\xi_{n}}\overline{u}_{n}\textbf{l}_{|\overline{u}_{n}|\leq M}\rightarrow \sqrt{\xi}\overline{u}\textbf{l}_{|\overline{u}|\leq M}\ \ strongly\ \ in\ L^{2}(0,T;L^{2}).
\end{equation}
Then
\begin{equation}\label{b32}
\begin{aligned}
&\int_{0}^{T}\int_{\Omega}|\sqrt{\xi_{n}}\overline{u}_{n}-\sqrt{\xi}\overline{u}|^{2}dxdzdt\\
&\leq 2\int_{0}^{T}\int_{\Omega}|\sqrt{\xi_{n}}\overline{u}_{n}\textbf{l}_{|\overline{u}_{n}|\leq M}-\sqrt{\xi}\overline{u}\textbf{l}_{|\overline{u}|\leq M}|^{2}dxdzdt\\
&\ \ +4\int_{0}^{T}\int_{\Omega}|\sqrt{\xi_{n}}\overline{u}_{n}\textbf{l}_{|\overline{u}_{n}|\geq M}|^{2}dxdzdt+4\int_{0}^{T}\int_{\Omega}|\sqrt{\xi}\overline{u}\textbf{l}_{|\overline{u}|\geq M}|^{2}dxdzdt\\
&\leq2\int_{0}^{T}\int_{\Omega}|\sqrt{\xi_{n}}\overline{u}_{n}\textbf{l}_{|\overline{u}_{n}|\leq M}-\sqrt{\xi}\overline{u}\textbf{l}_{|\overline{u}|\leq M}|^{2}dxdzdt\\
&\ \ +\frac{4}{M}\int_{0}^{T}\int_{\Omega}\xi_{n}|\overline{u}_{n}|^{3}dxdzdt+
\frac{4}{M}\int_{0}^{T}\int_{\Omega}\xi|\overline{u}|^{3}dxdzdt.\\
\end{aligned}
\end{equation}
Due to the Fatou's Lemma and H$\ddot{o}$lder inequality, we get
\begin{equation}\label{b33}
\begin{aligned}
&\int_{0}^{T}\int_{\Omega}\xi |\overline{u}|^{3}dxdzdt\leq\int_{0}^{T}\int_{\Omega}\liminf\xi_{n}|\overline{u}_{n}|^{3}dxdzdt\\
&\leq\liminf\int_{0}^{T}\int_{\Omega}\xi_{n}|\overline{u}_{n}|^{3}dxdzdt\\
&\leq\liminf\int_{0}^{T}\int_{\Omega}\xi_{n}|\frac{1}{h}\int_{0}^{h}u_{n}(\tau)d\tau|^{3}dxdzdt\\
&\leq\liminf\frac{1}{h}\int_{0}^{T}\int_{\Omega}\xi_{n}\int_{0}^{h}|u_{n}(\tau)|^{3}d\tau dxdzdt\\
&\leq\liminf\frac{1}{h}\int_{0}^{h}\int_{0}^{T}\int_{\Omega_{x}}\int_{0}^{h}\xi_{n}|u_{n}(\tau)|^{3}d\tau dxdtdz\\
&\leq\liminf\frac{E_{0}(\xi_{0},u_{0})}{h}\int_{0}^{h}dz\leq E_{0}(\xi_{0},u_{0}).
\end{aligned}
\end{equation}
Then (\ref{b33}) and (\ref{b31}) gives
$$
\parallel\sqrt{\xi_{n}}\overline{u}_{n}-\sqrt{\xi}\overline{u}\parallel_{L^{2}(0,T;L^{2}(\Omega))}\leq_{}\frac{C}{\sqrt{M}}
$$
for fixed $C>0$. we take $M\rightarrow\infty$ to get
$$\sqrt{\xi_{n}}\overline{u}_{n}\rightarrow\sqrt{\xi}\overline{u}\ \ strongly\ in\ L^{2}(0,T;L^{2}).$$
Similarly, for $\widetilde{u}=\frac{1}{h}\int_{0}^{h}u(\tau)d\tau$, we can obtain
$$\sqrt{\xi_{n}}\widetilde{u}_{n}\rightarrow\sqrt{\xi}\widetilde{u}\ \ strongly\ in\ L^{2}(0,T;L^{2}).$$
Thus the proof of this lemma is completed.
\end{proof}
\subsubsection{ Convergence of $\partial_{z}(\xi_{n}u_{n}w_{n})$.}
Let $\varphi \in C^{\infty}_{0}([0,T];\Omega)$ be a smooth function with compact support such as $ \varphi(T,x,z)=0$ and $ \varphi(0,x,z)=\varphi_{0}(x,z)$. With (\ref{b28}), we have

\begin{equation}\label{b34}
\begin{aligned}
&\int_{0}^{T}\int_{\Omega}\partial_{z}(\xi_{n} u_{n}w_{n})\cdot\varphi dxdzdt
=-\int_{0}^{T}\int_{\Omega}\xi_{n} u_{n}w_{n}\cdot \partial_{z}\varphi dxdzdt\\
&=-\int_{0}^{T}\int_{\Omega}u_{n}(-{\rm div}_{x}(\xi_{n}\widetilde{u}_{n})+z{\rm div}_{x}(\xi_{n}\overline{u}_{n}))\cdot \partial_{z}\varphi dxdzdt\\
&=-\int_{0}^{T}\int_{\Omega}(-{\rm div}_{x}(\xi_{n}\widetilde{u}_{n}\otimes u_{n})+\xi_{n}\widetilde{u}_{n}\cdot\nabla_{x}u_{n}+z
{\rm div}_{x}(\xi_{n}\overline{u}_{n}\otimes u_{n})\\
&\ \ \ -z\xi_{n}\overline{u}_{n}\cdot\nabla_{x}u_{n})\cdot \partial_{z}\varphi dxdzdt\\
&=-\int_{0}^{T}\int_{\Omega}\xi_{n}\widetilde{u}_{n}\otimes u_{n}:\partial_{z}\nabla_{x}\varphi dxdzdt+\int_{0}^{T}\int_{\Omega}\xi_{n}\overline{u}_{n}\otimes u_{n}:z\partial_{z}\nabla_{x}\varphi dxdzdt\\
&\ \ \ -\int_{0}^{T}\int_{\Omega}\xi_{n}\widetilde{u}_{n}\cdot\nabla_{x}u_{n}\cdot\partial_{z}\varphi dxdzdt
+\int_{0}^{T}\int_{\Omega}\xi_{n}\overline{u}_{n}\cdot\nabla_{x}u_{n}\cdot z\partial_{z}\varphi dxdzdt.
\end{aligned}
\end{equation}
Since
\begin{equation*}
\begin{aligned}
&\int_{0}^{T}\int_{\Omega}\sqrt{\xi_{n}}\nabla_{x}u_{n}:\varphi dxdzdt\\
&=\int_{0}^{T}\int_{\Omega}(\nabla_{x}(\sqrt{\xi_{n}}u_{n})-\nabla_{x}\sqrt{\xi_{n}}\otimes u_{n}):\varphi dxdzdt\\
&=-\int_{0}^{T}\int_{\Omega}(\sqrt{\xi_{n}}u_{n})\cdot{\rm div}\varphi dxdzdt-\int_{0}^{T}\int_{\Omega}\nabla_{x}\sqrt{\xi_{n}}\otimes u_{n}:\varphi dxdzdt\\
&\rightarrow-\int_{0}^{T}\int_{\Omega}(\sqrt{\xi}u)\cdot{\rm div}\varphi dxdzdt-\int_{0}^{T}\int_{\Omega}\nabla_{x}\sqrt{\xi}\otimes u:\varphi dxdzdt\\
&=\int_{0}^{T}\int_{\Omega}\sqrt{\xi}\nabla_{x}u:\varphi dxdzdt,
\end{aligned}
\end{equation*}
as $n\rightarrow\infty$.
Hence
$$
\sqrt{\xi_{n}}\nabla_{x}u_{n}\rightarrow \sqrt{\xi}\nabla_{x}u \ \ weakly \ in\ \  L^{2}(0,T;L^{2}),
$$
and combines with (\ref{b34}), we have
$$\int_{0}^{T}\int_{\Omega}\partial_{z}(\xi_{n} u_{n}w_{n})\cdot\varphi dxdzdt
\rightarrow\int_{0}^{T}\int_{\Omega}\partial_{z}(\xi uw)\cdot\varphi dxdzdt
$$
as $n\rightarrow\infty$,
where
$$
w=-\frac{{\rm div}_{x}(\xi\widetilde{u})}{\xi}+z\frac{{\rm div}_{x}(\xi\overline{u})}{\xi}.
$$

\subsubsection{ Convergence of nonlinear diffusion terms.}

By direct computation, we have
\begin{equation}\label{b35}
\begin{aligned}
&\int_{0}^{T}\int_{\Omega} {\rm div}_{x}(\xi_{n}D_{x}(u_{n}))\varphi dxdzdt=-\int_{0}^{T}\int_{\Omega} (\xi_{n}D_{x}(u_{n})):\nabla_{x}\varphi dxdzdt\\
&=\frac{1}{2}\int_{0}^{T}\int_{\Omega}(\xi_{n}u_{n}\cdot\Delta_{x}\varphi+2\nabla_{x}\varphi\cdot\nabla_{x}\sqrt{\xi_{n}}
\cdot\sqrt{\xi_{n}}u_{n})dxdzdt\\
&\ \ \ +\frac{1}{2}\int_{0}^{T}\int_{\Omega}(\xi_{n}u_{n}\cdot{\rm div}_{x}(\nabla_{x}^{t}\varphi)+2\nabla_{x}^{t}\varphi\cdot\nabla_{x}\sqrt{\xi_{n}}
\cdot\sqrt{\xi_{n}}u_{n})dxdzdt.\\
\end{aligned}
\end{equation}
Since $\sqrt{\xi_{n}} \rightarrow \sqrt{\xi}\ \ strongly\ \ in\ L^{2}([0,T];H^{1}),\ \xi_{n}u_{n} \rightarrow \xi u\ \ strongly\ \ in\ L^{2}(0,T;L^{2})$, $\sqrt{\xi_{n}}u_{n}\rightarrow \sqrt{\xi}u\ \ strongly\ \ in\ L^{2}(0,T;L^{2})$, we obtain
\begin{equation}\label{b36}
\begin{aligned}
&\frac{1}{2}\int_{0}^{T}\int_{\Omega}(\xi_{n}u_{n}\cdot\Delta_{x}\varphi+2\nabla_{x}\varphi\cdot\nabla_{x}\sqrt{\xi_{n}}
\cdot\sqrt{\xi_{n}}u_{n})dxdzdt\\
&+\frac{1}{2}\int_{0}^{T}\int_{\Omega}(\xi_{n}u_{n}\cdot{\rm div}_{x}(\nabla_{x}^{t}\varphi)+2\nabla_{x}^{t}\varphi\cdot\nabla_{x}\sqrt{\xi_{n}}
\cdot\sqrt{\xi_{n}}u_{n})dxdzdt\\
\rightarrow&\frac{1}{2}\int_{0}^{T}\int_{\Omega}(\xi u\cdot\Delta_{x}\varphi+2\nabla_{x}\varphi\cdot\nabla_{x}\sqrt{\xi}
\cdot\sqrt{\xi}u)dxdzdt\\
&+\frac{1}{2}\int_{0}^{T}\int_{\Omega}(\xi u\cdot{\rm div}_{x}(\nabla_{x}^{t}\varphi)+2\nabla_{x}^{t}\varphi\cdot\nabla_{x}\sqrt{\xi}
\cdot\sqrt{\xi}u)dxdzdt.\\
\end{aligned}
\end{equation}
as $n\rightarrow\infty$.
Hence
$$
\int_{0}^{T}\int_{\Omega} {\rm div}_{x}(\xi_{n}D_{x}(u_{n}))\varphi dxdzdt\rightarrow\int_{0}^{T}\int_{\Omega} {\rm div}_{x}(\xi D_{x}(u))\varphi dxdzdt,
$$
when $n\rightarrow\infty$.
Similarly, we have
\begin{equation}\label{b37}
\begin{aligned}
\int_{0}^{T}\int_{\Omega}\xi_{n}\nabla_{x} \Delta_{x}^{5}\xi_{n}\varphi dxdzdt&=-\int_{0}^{T}\int_{\Omega}(\xi_{n}{\rm div}_{x}\varphi+\varphi\cdot\nabla_{x} \xi_{n})\Delta_{x}^{5}\xi_{n}dxdzdt\\
&=-\int_{0}^{T}\int_{\Omega}\Delta_{x}^{2}(\xi_{n}{\rm div}_{x}\varphi+\varphi\cdot\nabla_{x} \xi_{n})\Delta_{x}^{3}\xi_{n}dxdzdt.
\end{aligned}
\end{equation}
 The limit of the term $-\int_{0}^{T}\int_{\Omega}\varphi\cdot\nabla_{x}\Delta_{x}^{2}\xi_{n}\Delta_{x}^{3}\xi_{n}dxdzdt$ is difficult.
Due to  $\xi_{n}\rightarrow \xi$ strongly in $C([0,T];H^{5})$ and $\xi_{n} \rightharpoonup \xi$ weakly in $ L^{2}(0,T;H^{6})$, we have
$$-\int_{0}^{T}\int_{\Omega}\varphi\cdot\nabla_{x}\Delta_{x}^{2} \xi_{n}\Delta_{x}^{3}\xi_{n}dxdzdt
\rightarrow -\int_{0}^{T}\int_{\Omega}\varphi\cdot\nabla_{x}\Delta_{x}^{2} \xi\Delta_{x}^{3}\xi dxdzdt,$$ as $n\rightarrow\infty$.
Likewise, applying the above arguments, we can handle with the other terms of
$$-\int_{0}^{T}\int_{\Omega}\Delta_{x}^{2}(\xi_{n}{\rm div}_{x}\varphi+\varphi\cdot\nabla_{x} \xi_{n})\Delta_{x}^{3}\xi_{n}dxdzdt.$$
Thus we have
$$\int_{0}^{T}\int_{\Omega}\xi_{n}\nabla_{x} \Delta_{x}^{5}\xi_{n}\varphi dxdzdt\rightarrow\int_{0}^{T}\int_{\Omega}\xi\nabla_{x} \Delta_{x}^{5}\xi\varphi dxdzdt,$$
as $n\rightarrow\infty$.

To deal with the quantum term, we use the same arguments to obtain
\begin{equation*}
\begin{aligned}
&\int_{0}^{T}\int_{\Omega}\xi_{n}\nabla_{x}(\dfrac{\Delta_{x}\sqrt{\xi_{n}}}{\sqrt{\xi_{n}}})\varphi dxdzdt\\
&=-2\int_{0}^{T}\int_{\Omega}\Delta_{x}\sqrt{\xi_{n}}\nabla_{x}\sqrt{\xi_{n}}dxdzdt-\int_{0}^{T}\int_{\Omega}
\Delta_{x}\sqrt{\xi_{n}}\sqrt{\xi_{n}}{\rm div}_{x}\varphi dxdzdt\\
&\rightarrow-2\int_{0}^{T}\int_{\Omega}\Delta_{x}\sqrt{\xi}\nabla_{x}\sqrt{\xi}dxdzdt-\int_{0}^{T}\int_{\Omega}
\Delta_{x}\sqrt{\xi}\sqrt{\xi}{\rm div}_{x}\varphi dxdzdt\\
&=\int_{0}^{T}\int_{\Omega}\xi\nabla_{x}(\dfrac{\Delta_{x}\sqrt{\xi}}{\sqrt{\xi}})\varphi dxdzdt.
\end{aligned}
\end{equation*} as $n\rightarrow\infty$.

With the above compactness results in hand, we can take the limits as $n\rightarrow\infty$ in the approximate system (\ref{b3}) and (\ref{b8}). Thus, we can show that $(\xi,u,w)$ solves
$$\xi_{t}+{\rm div}_{x}(\rho \overline{u})=\varepsilon\Delta_{x}\rho,\ \ on\ (0,T)\times\Omega.$$
And for any test function $\varphi$, the following identity holds
\begin{equation}\label{b38}
\begin{aligned}
&\int_{\Omega}\xi u(T)\varphi dxdz-\int_{\Omega} m_{0}\varphi dxdz+\mu\int_{0}^{T}\int_{\Omega} \Delta u\cdot\Delta \varphi dxdzdt\\
&-\int_{0}^{T}\int_{\Omega}\xi uw\partial_{z}\varphi dxdzdt
-\int_{0}^{T}\int_{\Omega} (\xi u\otimes u)\cdot\nabla_{x} \varphi dxdzdt\\
&+2\overline{\nu}_{1}\int_{0}^{T}\int_{\Omega} \xi D_{x}(u)\cdot\nabla_{x} \varphi dxdzdt-\int_{0}^{T}\int_{\Omega} \xi{\rm div}_{x} \varphi dxdzdt\\
&+\overline{\nu}_{2}\int_{0}^{T}\int_{\Omega}\xi\partial_{z}u\partial_{z}\varphi dxdzdt+\eta\int_{0}^{T}\int_{\Omega} \xi^{-10}{\rm div}_{x}\varphi dxdzdt\\
&+\varepsilon\int_{0}^{T}\int_{\Omega} \nabla_{x}\xi\cdot\nabla_{x} u\varphi dxdzdt
+r_{0}\int_{0}^{T}\int_{\Omega} u\varphi dxdzdt\\
&+r\int_{0}^{T}\int_{\Omega} \xi |u|u \varphi dxdt-\delta\int_{0}^{T}\int_{\Omega} \xi \nabla_{x}\Delta_{x}^{5} \xi\varphi dxdzdt \\
&=-2\kappa\int_{0}^{T}\int_{\Omega}\Delta_{x}\sqrt{\xi}\nabla_{x}\sqrt{\xi}\varphi dxdzdt-\kappa\int_{0}^{T}\int_{\Omega}\Delta_{x}\sqrt{\xi}\sqrt{\xi}{\rm div}_{x}\varphi dxdzdt,
\end{aligned}
\end{equation}
where
$$
w=-\frac{{\rm div}_{x}(\xi\widetilde{u})}{\xi}+z\frac{{\rm div}_{x}(\xi\overline{u})}{\xi}.
$$
Thanks to the lower semicontinuity of norms, we can pass to the limits in the energy estimate (\ref{b20}), and we have the following energy inequality in the sense of distributions on $(0,T)$.
\begin{equation}\label{b39}
\begin{aligned}
&\sup_{t\in(0,T)}E(\xi,u)+
4\epsilon\int_{0}^{T}\int_{\Omega}|\nabla_{x}\sqrt{\xi}|^{2}dxdzdt+r_{0}\int_{0}^{T}\int_{\Omega}u^{2}dxdzdt\\
&+r\int_{0}^{T}\int_{\Omega}\xi|u|^{3}dxdzdt
+\int_{0}^{T}\int_{\Omega} 2\overline{\nu}_{1}\xi|D_{x}(u)|^{2}dxdzdt\\
&+\int_{0}^{T}\int_{\Omega}\overline{\nu}_{2}\xi|\partial_{z}u|^{2}dxdzdt+\mu\int_{0}^{T}\int_{\Omega} |\Delta u|^{2}dxdzdt\\
&+\frac{2}{5}\eta\epsilon\int_{0}^{T}\int_{\Omega}|\nabla_{x}\xi^{-5}|^{2}dxdzdt+\frac{\kappa\epsilon}{2}\int_{0}^{T}\int_{\Omega}\xi|\nabla_{x}\ln\xi|^{2}dxdzdt\\
&+\delta\epsilon\int_{0}^{T}\int_{\Omega}|\Delta^{3}_{x}\xi|^{2}dxdzdt
\leq E(\xi_{0},u_{0}),
\end{aligned}
\end{equation}
where
$$
E(\xi,u)=\int_{\Omega}\frac{1}{2}\xi u^{2}+\xi\ln\xi-\xi+1+\frac{1}{11}\eta\xi^{-10}+\kappa|\nabla_{x}\sqrt{\xi}|^{2}+
\frac{\delta}{2}|\nabla_{x}\Delta_{x}^{2}\xi|^{2}dxdz.
$$

Thus, we have the following proposition on the existence of weak solutions at this level approximate system.
\begin{Proposition}\label{pro3}
There exists a weak solution to the following system
\begin{equation}\label{b40}
\left\{\begin{array}{lll}
\partial_{t}\xi+{\rm div}_{x}(\xi \overline{u})=\varepsilon\Delta_{x}\xi,\\
\partial_{t}(\xi u)+{\rm div}_{x}(\xi u\otimes u)+\partial_{z}(\xi uw)+\nabla_{x}\xi+r_{0} u+r\xi |u|u+\mu\Delta^{2}u\\
=2\overline{\nu}_{1}{\rm div}_{x}(\xi D_{x}(u))+\overline{\nu}_{2}\partial_{z}(\xi \partial_{z}u)+\varepsilon\nabla_{x}\cdot\nabla_{x}u+
\eta\nabla_{x}\xi^{-10}+\\ \ \ \  \kappa\xi\nabla_{x}(\dfrac{\Delta_{x}\sqrt{\xi}}{\sqrt{\xi}})+\delta\xi\nabla_{x}\Delta_{x}^{5}\xi\\
\partial_{z}\xi=0.
\end{array}\right.
\end{equation}
with suitable initial data, for any $T>0$. In particular, the weak solutions $(\xi,u,w)$ satisfy the energy inequality (\ref{b39}).
\end{Proposition}
\section{B-D entropy and passing to the limits as $\varepsilon,\mu\rightarrow 0$}
In this section, the B-D entropy estimate for the approximate system in proposition \ref{pro3} will be obtained which was first introduced by Bresch and Desjardins in \cite{dbcl2003}. By (\ref{b26}) and (\ref{b21}), we have
\begin{equation}\label{c1}
\begin{aligned}
\ \xi(x,t)\geq C(\delta,\eta)>0,\ \ \xi\in L^{\infty}(0,T;H^{5})\cap L^{2}(0,T;H^{6}).
\end{aligned}
\end{equation}
\subsection{B-D entropy.}
Thanks to (\ref{c1}),
 we can use $\frac{\nabla\xi}{\xi}$ as a test function to test the momentum equation to derive the B-D entropy. To this end, we first take the gradient of the mass equation and multiply by $2\overline{\nu}_{1}$. Next we sum the obtained equation with the momentum  of system (\ref{b40})  to get
 \begin{equation}\label{c2}
\begin{aligned}
&\partial_{t}(\xi(u+2\overline{\nu}_{1}\nabla_{x}\ln\xi)+{\rm div}_{x}(\xi(u+2\overline{\nu}_{1}\nabla_{x}\ln\xi)\otimes u)+\partial_{z}(2\overline{\nu}_{1}\nabla_{x}(\xi w))+\partial_{z}(\xi uw)\\
&+\nabla_{x}\xi+r_{0}u+r\xi|u|u+\mu\Delta^{2}u={\rm div}_{x}(2\overline{\nu}_{1}\xi A_{x}(u))+\overline{\nu}_{2}\partial_{z}(\xi\partial_{z}u)
+\eta\nabla_{x}\xi^{-10}\\
&-\varepsilon\nabla_{x}\xi\cdot\nabla_{x}u+2\overline{\nu}_{1}\varepsilon\nabla_{x}\Delta_{x}\xi+
\kappa\xi\nabla_{x}(\dfrac{\Delta_{x}\sqrt{\xi}}{\sqrt{\xi}})+\delta\xi\nabla_{x}\Delta_{x}^{5}\xi\\
\end{aligned}
\end{equation}
 where $A_{x}(u)=\dfrac{\nabla_{x}u-\nabla_{x}^{t}u}{2}$.
\begin{Lemma}\label{l3} Under the assumption $(\ref{c1})$, we have the following B-D entropy
\begin{equation}\label{c3}
\begin{aligned}
&\int_{\Omega}(\frac{1}{2}\xi(u+2\overline{\nu}_{1} \nabla_{x}\ln \xi)^{2}-2\overline{\nu}_{1}r_{0}\ln\xi) dxdzdt
+2\overline{\nu}_{1}\int_{0}^{T}\int_{\Omega}\xi|\partial_{z}w|^{2}dxdzdt\\
&+8\overline{\nu}_{1}\int_{0}^{T}\int_{\Omega}
|\nabla_{x}\sqrt{\xi}|^{2}dxdzdt +r\int_{0}^{T}\int_{\Omega}\xi|u|^{3}dxdzdt+\overline{\nu}_{2}\int_{0}^{T}\int_{\Omega}\xi|\partial_{z}u|^{2}dxdzdt\\
&+\frac{8}{5}\eta\overline{\nu}_{1}\int_{0}^{T}\int_{\Omega}|\nabla_{x}\xi^{-5}|^{2}dxdzdt
+\mu\int_{0}^{T}\int_{\Omega}|\Delta u|^{2}dxdzdt\\
&+\kappa\overline{\nu}_{1}\int_{0}^{T}\int_{\Omega}\xi|\nabla_{x}^{2}\ln\xi|^{2}dxdzdt
+r_{0}\int_{0}^{T}\int_{\Omega}u^{2}dxdzdt
+2\delta\overline{\nu}_{1}\int_{0}^{T}\int_{\Omega}|\Delta_{x}^{3}\xi|^{2}dxdzdt\\
&+2\overline{\nu}_{1}\int_{0}^{T}\int_{\Omega}\xi|A_{x}(u)|^{2}dxdzdt
+\overline{\nu}_{1}r_{0}\varepsilon\int_{0}^{T}\int_{\Omega}\frac{|\nabla_{x}\xi|^{2}}{\xi^{2}}dxdzdt\\
&\leq\int_{\Omega}(\frac{1}{2}\xi_{0}(u_{0}+2\overline{\nu}_{1} \nabla_{x}\ln \xi_{0})^{2}
-2\overline{\nu}_{1}r_{0}\ln\xi_{0}) dxdz+E_{0}+C+\varepsilon C(\delta,\eta)+\sqrt{\mu}C(\delta,\eta)
\end{aligned}
\end{equation}

where $C$ is a generic positive constant depending on the initial data and other constants but independent of $\varepsilon,\delta,\eta,r_{0},\mu$, and $C(\delta,\eta)$is a generic positive constant  only depending on
$\delta,\eta$.
\end{Lemma}
\begin{proof}
Multiplying (\ref{c2}) by $u+2\overline{\nu}_{1}\nabla_{x}\ln\xi$ and integrating over $\Omega$, we have
\begin{equation*}
\begin{aligned}
&\frac{d}{dt}\int_{\Omega}(\frac{1}{2}\xi(u+2\overline{\nu}_{1} \nabla_{x}\ln \xi)^{2} +(\xi\ln\xi-\xi+1)+\frac{1}{11}\eta\xi^{-10}+\kappa|\nabla_{x}\sqrt{\xi}|^{2}\\&+\frac{\delta}{2}|\nabla_{x}\Delta_{x}^{2}\xi|^{2}dxdz
+2\overline{\nu}_{1}\int_{\Omega}\xi|\partial_{z}w|^{2}dxdz+(8\overline{\nu}_{1}-4\varepsilon)\int_{\Omega}|\nabla_{x}\sqrt{\xi}|^{2}dxdz \\&+r\int_{\Omega}\xi|u|^{3}dxdz+\overline{\nu}_{2}\int_{\Omega}\xi|\partial_{z}u|^{2}dxdz
+\frac{2}{5}\eta(\varepsilon+4\overline{\nu}_{1})\int_{\Omega}|\nabla_{x}\xi^{-5}|^{2}dxdz\\
&+\mu\int_{\Omega}|\Delta_{x}u|^{2}dxdz+\frac{\kappa(\varepsilon+2\overline{\nu}_{1})}{2}\int_{\Omega}\xi|\nabla_{x}^{2}\ln\xi|^{2}dxdz
+\delta(\varepsilon+2\overline{\nu}_{1})\int_{\Omega}|\Delta_{x}^{3}\xi|^{2}dxdz\\&+2\overline{\nu}_{1}\int_{\Omega}\xi|A_{x}(u)|^{2}dxdz
+r_{0}\int_{\Omega}u^{2}dxdz\\
&=2\overline{\nu}_{1}\int_{\Omega}{\rm div}_{x}(2\overline{\nu}_{1}\xi A_{x}(u))\nabla_{x}\ln\xi dxdz
-2\overline{\nu}_{1}\mu\int_{\Omega}\Delta_{x} u\nabla_{x}\Delta_{x}\ln\xi dxdz\\
\end{aligned}
\end{equation*}
\begin{equation}\label{c4}
\begin{aligned}
&\ \ \ +2\overline{\nu}_{1}\varepsilon\int_{\Omega}\nabla_{x}\xi\cdot\nabla_{x}^{2}\ln\xi\cdot udxdz+
4\overline{\nu}_{1}^{2}\varepsilon\int_{\Omega}\nabla_{x}\xi\cdot\nabla_{x}^{2}\ln\xi\cdot\nabla_{x}\ln\xi dxdz\\
&\ \ \ +2\overline{\nu}_{1}\varepsilon\int_{\Omega}\nabla_{x}\Delta_{x}\xi(u+2\overline{\nu}_{1} \nabla_{x}\ln \xi)dxdz-2\overline{\nu}_{1}r\int_{\Omega}|u|u\nabla_{x}\xi dxdz\\
&\ \ \ -2\overline{\nu}_{1}r_{0}\int_{\Omega}u\frac{\nabla_{x}\xi}{\xi}dxdz
=\sum_{i=1}^{7}I_{i}.
\end{aligned}
\end{equation}

We claim that $I_{1}=0$ due to the periodic conditions. That is
\begin{equation}\label{c5}
\begin{aligned}
&2\int_{\Omega}{\rm div}_{x}(\xi A_{x}(u))\nabla_{x}\ln\xi dxdz=\int_{\Omega}\partial_{i}[\xi(\partial_{i}u_{j}-\partial_{j}u_{i})]\frac{\partial_{j}\xi}{\xi}dxdz\\
&=\int_{\Omega}\frac{\partial_{i}\xi\partial_{j}\xi}{\xi}(\partial_{i}u_{j}-\partial_{j}u_{i})
+(\partial_{i}\partial_{i}u_{j}-\partial_{i}\partial_{j}u_{i})\partial_{j}\xi dxdz\\
&=\int_{\Omega}\partial_{i}\partial_{i}u_{j}\partial_{j}\xi-\partial_{i}\partial_{j}u_{i}\partial_{j}\xi dxdz\\
&=\int_{\Omega}\partial_{i}\partial_{i}u_{j}\partial_{j}\xi-\partial_{j}\partial_{j}u_{i}\partial_{i}\xi dxdz=0,
\end{aligned}
\end{equation}

then $I_{1}=0$. For $I_{7}$, we have
\begin{equation}\label{c6}
\begin{aligned}
I_{7}&=-2\overline{\nu}_{1}r_{0}\int_{\Omega}\frac{{\rm div}_{x}(\xi u)}{\xi}dxdz
=2\overline{\nu}_{1}r_{0}\int_{\Omega}\frac{\partial_{t}\xi+\partial_{z}(\xi w)-\varepsilon\Delta_{x}\xi}{\xi}dxdz\\
&=2\overline{\nu}_{1}r_{0}\int_{\Omega}\partial_{t}\ln\xi dxdz-2\overline{\nu}_{1}r_{0}\varepsilon\int_{\Omega}\frac{\Delta_{x}\xi}{\xi}dxdz\\
&=2\overline{\nu}_{1}r_{0}\int_{\Omega}\partial_{t}\ln\xi dxdz-2\overline{\nu}_{1}r_{0}\varepsilon\int_{\Omega}\frac{|\nabla_{x}\xi|^{2}}{\xi^{2}}dxdz.
\end{aligned}
\end{equation}

Substituting (\ref{c4})-(\ref{c6}) into (\ref{c4}) and integrating it with respect to the time $t$ over [0,T], we have
\begin{equation}\label{c7}
\begin{aligned}
&\int_{\Omega}(\frac{1}{2}\xi(u+2\overline{\nu}_{1} \nabla_{x}\ln \xi)^{2}-2\overline{\nu}_{1}r_{0}\ln\xi) dxdzdt
+2\overline{\nu}_{1}\int_{0}^{T}\int_{\Omega}\xi|\partial_{z}w|^{2}dxdzdt\\
&+8\overline{\nu}_{1}\int_{0}^{T}\int_{\Omega}
|\nabla_{x}\sqrt{\xi}|^{2}dxdzdt +r\int_{0}^{T}\int_{\Omega}\xi|u|^{3}dxdzdt+\overline{\nu}_{2}\int_{0}^{T}\int_{\Omega}\xi|\partial_{z}u|^{2}dxdzdt\\
&+\frac{8}{5}\eta\overline{\nu}_{1}\int_{0}^{T}\int_{\Omega}|\nabla_{x}\xi^{-5}|^{2}dxdzdt
+\mu\int_{0}^{T}\int_{\Omega}|\Delta u|^{2}dxdzdt+\\
&\kappa\overline{\nu}_{1}\int_{0}^{T}
\int_{\Omega}\xi|\nabla_{x}^{2}\ln\xi|^{2}dxdzdt
+r_{0}\int_{0}^{T}\int_{\Omega}u^{2}dxdzdt
+2\delta\overline{\nu}_{1}\int_{0}^{T}\int_{\Omega}|\Delta_{x}^{3}\xi|^{2}dxdzdt\\
&+2\overline{\nu}_{1}\int_{0}^{T}\int_{\Omega}\xi|A_{x}(u)|^{2}dxdzdt
+2\overline{\nu}_{1}r_{0}\varepsilon\int_{0}^{T}\int_{\Omega}\frac{|\nabla_{x}\xi|^{2}}{\xi^{2}}dxdzdt\\
&\leq\int_{\Omega}(\frac{1}{2}\xi_{0}(u_{0}+2\overline{\nu}_{1} \nabla_{x}\ln \xi_{0})^{2}
-2\overline{\nu}_{1}r_{0}\ln\xi_{0}) dxdz+\int_{0}^{T}\sum_{i=2}^{6}I_{i}dt+E_{0},
\end{aligned}
\end{equation}
where we have used the energy inequality (\ref{b39}). Next we control the rest terms on the right hand of (\ref{c7}).
\begin{equation}\label{c8}
\begin{aligned}
\int&_{0}^{T}I_{2}dt=-2\overline{\nu}_{1}\mu\int_{0}^{T}\int_{\Omega}\Delta_{x}u\cdot\nabla_{x}\Delta_{x}\ln\xi dxdzdt\\
&=-2\overline{\nu}_{1}\mu\int_{0}^{T}\int_{\Omega}\Delta_{x}u\cdot(\frac{\nabla_{x}\Delta_{x}\xi}{\xi}-
\frac{\Delta_{x}\xi\nabla_{x}\xi}{\xi^{2}}-2\frac{(\nabla_{x}\xi\cdot\nabla_{x})\nabla_{x}\xi}{\xi^{2}}\\
&\ \ \ +2\frac{|\nabla_{x}\xi|^{2}\nabla_{x}\xi}{\xi^{3}}) dxdzdt\\
&\leq C\mu\int_{0}^{T}\int_{\Omega}|\Delta_{x} u|\big[\xi^{-1}|\nabla_{x}^{3}\xi|+\xi^{-2}|\nabla_{x}\xi||\nabla_{x}^{2}\xi|+\xi^{-3}
|\nabla_{x}\xi|^{3}\big]dxdzdt\\
&\leq C\sqrt{\mu}\|\sqrt{\mu}\Delta_{x} u\|_{L^{2}(L^{2})}\big[\|\xi^{-1}\|_{L^{\infty}(L^{\infty})}\|\nabla_{x}^{3}\xi\|_{L^{2}(L^{2})}\\
&\ \ \ +\|\xi^{-1}\|_{L^{\infty}(L^{\infty})}^{2}\|\nabla_{x}\xi\|_{L^{\infty}(L^{\infty})}\|\nabla_{x}
^{2}\xi\|_{L^{2}(L^{2})}
+\|\xi^{-1}\|_{L^{\infty}(L^{\infty})}^{3}\|\nabla_{x}\xi\|_{L^{6}
(L^{6})}^{3}\big]\\
&\leq C\sqrt{\mu}\|\sqrt{\mu}\Delta_{x} u\|_{L^{2}(L^{2})}\big[\|\xi^{-1}\|_{L^{\infty}(L^{\infty})}^{3}\|\nabla_{x}^{3}\xi\|_{L^{2}(L^{2})}
+\|\xi^{-1}\|_{L^{\infty}(L^{\infty})}\big]\\
&\leq C(\delta,\eta)\sqrt{\mu},
\end{aligned}
\end{equation}
\begin{equation}\label{c9}
\begin{aligned}
\int_{0}^{T}I_{4}dt&=4\overline{\nu}_{1}^{2}\varepsilon\int_{0}^{T}\int_{\Omega}\nabla_{x}\xi\cdot\nabla_{x}^{2}\ln\xi\cdot\nabla_{x}\ln\xi dxdzdt\\
&=-2\overline{\nu}_{1}^{2}\varepsilon\int_{0}^{T}\int_{\Omega}\Delta_{x}\xi|\nabla_{x}\ln\xi|^{2} dxdzdt\\
&\leq2\overline{\nu}_{1}^{2}\varepsilon\parallel\frac{1}{\xi}\parallel^{2}_{L^{\infty}(L^{\infty})}
\parallel\nabla_{x}\xi\parallel_{L^{\infty}(L^{\infty})}\parallel\nabla_{x}\xi\parallel_{L^{2}(L^{2})}\parallel
\Delta_{x}\xi\parallel_{L^{2}(L^{2})}\\
&\leq\varepsilon C(\delta,\eta).
\end{aligned}
\end{equation}
Besides this, we also have
\begin{equation}\label{c10}
\begin{aligned}
\int_{0}^{T}I_{3}dt&=2\overline{\nu}_{1}\varepsilon\int_{0}^{T}\int_{\Omega}\nabla_{x}\xi\cdot\nabla_{x}^{2}\ln\xi\cdot udxdzdt\\
&=-\overline{\nu}_{1}\varepsilon\int_{0}^{T}\int_{\Omega}|\frac{\nabla_{x}\xi}{\xi}|^{2}{\rm div}_{x}(\xi u)dxdzdt\\
&\leq\overline{\nu}_{1}\varepsilon\int_{0}^{T}\int_{\Omega}\frac{(\nabla_{x}\xi)^{2}}{\xi^{1.5}}\sqrt{\xi}\nabla_{x} u+\frac{(\nabla_{x}\xi)^{2}}{\xi^{2.5}}\sqrt{\xi}u\cdot\nabla_{x}\xi dxdzdt\\
&\leq \varepsilon \overline{\nu}_{1}\parallel\frac{1}{\xi}\parallel^{1.5}_{L^{\infty}(L^{\infty})}\parallel\nabla_{x}\xi\parallel_{L^{\infty}(L^{\infty})}
\parallel\nabla_{x}\xi\parallel_{L^{2}(L^{2})}\parallel\sqrt{\xi}\nabla_{x}u\parallel_{L^{2}(L^{2})}\\
&\ \ \ +\varepsilon \overline{\nu}_{1}\parallel\frac{1}{\xi}\parallel^{2.5}_{L^{\infty}(L^{\infty})}\parallel\nabla_{x}\xi\parallel^{2}_{L^{\infty}(L^{\infty})}
\parallel\nabla_{x}\xi\parallel_{L^{2}(L^{2})}\parallel\sqrt{\xi}u\parallel_{L^{2}(L^{2})}\\
&\leq\varepsilon\overline{\nu}_{1} C(\delta,\eta)(\parallel\sqrt{\xi}\nabla_{x}u\parallel_{L^{2}(L^{2})}+1)\\
&\leq \frac{\overline{\nu}_{1}}{2}\parallel\sqrt{\xi} A_{x}u\parallel_{L^{2}(L^{2})}^{2}+\varepsilon C(\delta,\eta),
\end{aligned}
\end{equation}
where we have used
$$\parallel\sqrt{\xi} A_{x}u\parallel_{L^{2}(L^{2})}^{2}+\parallel\sqrt{\xi} D_{x}u\parallel_{L^{2}(L^{2})}^{2}=\parallel\sqrt{\xi} \nabla_{x}u\parallel_{L^{2}(L^{2})}^{2}.
$$
And
\begin{equation}\label{c11}
\begin{aligned}
\int_{0}^{T}I_{5}dt&=2\overline{\nu}_{1}\varepsilon\int_{\Omega}\nabla_{x}\Delta_{x}\xi(u+2\overline{\nu}_{1} \nabla_{x}\ln \xi)dxdz\\
&\leq2\overline{\nu}_{1}\varepsilon\parallel\nabla_{x}\Delta_{x}\xi\parallel_{L^{2}(L^{2})}\parallel\sqrt{\xi}u\parallel_{L^{2}(L^{2})}
\parallel\sqrt{\xi}\parallel_{L^{\infty}(L^{\infty})}\\
&\ \ \ +4\overline{\nu}_{1}^{2}\varepsilon\parallel\nabla_{x}\Delta_{x}\xi\parallel_{L^{2}(L^{2})}
\parallel\nabla_{x}\xi\parallel_{L^{2}(L^{2})}
\parallel\xi^{-1}\parallel_{L^{\infty}(L^{\infty})}\\
&\leq \varepsilon C(\delta,\eta),
\end{aligned}
\end{equation}
\begin{equation}\label{c12}
\begin{aligned}
\int_{0}^{T}I_{6}dt&=-2\overline{\nu}_{1}r\int_{0}^{T}\int_{\Omega} |u|u\cdot\nabla_{x}\xi dxdzdt\\
&=2\overline{\nu}_{1}r\int_{0}^{T}\int_{\Omega} \xi(|u|{\rm div}_{x}u+u\frac{u}{|u|}\nabla_{x} u)dxdzdt\\
&\leq C\|\sqrt{\rho}u\|_{L^{2}}\|\sqrt{\rho}\nabla_{x} u\|_{L^{2}}\leq C+\frac{\overline{\nu}_{1}}{2}\parallel\sqrt{\xi}
A_{x}(u)\parallel_{L^{2}(L^{2})}^{2}.
\end{aligned}
\end{equation}

Then substituting (\ref{c8})-(\ref{c12}) into (\ref{c7}),  we obtain the B-D entropy (\ref{c3}).
\end{proof}

\subsection{Passing to the limits as $\mu,\varepsilon\rightarrow0$.} We denote the solutions to (\ref{b40}) at this level of approximation as $(\xi_{\mu,\varepsilon},u_{\mu,\varepsilon},w_{\mu,\varepsilon})$. From the energy inequality (\ref{b39}), it is easy to obtain the following uniform regularities.
\begin{equation}\label{c13}
\begin{aligned}
&\sqrt{\xi_{\mu,\varepsilon}}u_{\mu,\varepsilon}\in L^{\infty}(0,T;L^{2}),\eta\xi_{\mu,\varepsilon}^{-10}\in L^{\infty}(0,T;L^{1}),\xi_{n}^{\frac{1}{3}}u_{\mu,\varepsilon}\in L^{3}(0,T;L^{3}).\\
&\sqrt{\xi_{\mu,\varepsilon}}D_{x}(u_{\mu,\varepsilon})\in L^{2}(0,T;L^{2}),\sqrt{\xi_{\mu,\varepsilon}}\partial_{z}u_{\mu,\varepsilon}\in L^{2}(0,T;L^{2}),\sqrt{\mu}\Delta u_{\mu,\varepsilon}\in L^{2}(0,T;L^{2}),\\
&\sqrt{\delta}\xi_{\mu,\varepsilon}\in L^{\infty}(0,T;H^{5}),\sqrt{r_{0}}u_{\mu,\varepsilon}\in L^{2}(0,T;L^{2}),
\xi_{\mu,\varepsilon}\ln\xi_{\mu,\varepsilon}-\xi_{\mu,\varepsilon}+1\in L^{\infty}(0,T;L^{1}).
\end{aligned}
\end{equation}
The B-D entropy (\ref{c3}) gives the following additionally uniform regularities
\begin{equation}\label{c14}
\begin{aligned}
&\nabla_{x}\sqrt{\xi_{\mu,\varepsilon}}\in L^{\infty}(0,T;L^{2}),\ \sqrt{\delta}\xi_{\mu,\varepsilon}\in L^{2}(0,T;H^{6}),\ \sqrt{\xi_{\mu,\varepsilon}}\partial_{z}w\in L^{2}(0,T;L^{2}),\\
&\sqrt{\eta}\nabla\xi_{\mu,\varepsilon}^{-5}\in L^{2}(0,T;L^{2}), \sqrt{\xi_{\mu,\varepsilon}}A_{x} (u_{\mu,\varepsilon})\in L^{2}(0,T;L^{2}).
\end{aligned}
\end{equation}
Similar to the proof of Lemma \ref{lemma1}, we have the following uniform boundedness
$$
\sqrt{\kappa}\sqrt{\xi_{\mu,\varepsilon}}\in L^{2}(0,T;H^{2}),
 \sqrt[4]{\kappa}\nabla_{x}\xi_{\mu,\varepsilon}^{\frac{1}{4}}\in L^{4}(0,T;L^{4}).
 $$

Based on the above regularities, we have the following uniform compactness results.
\begin{Lemma}\label{l4}
Let $(\xi_{\mu,\varepsilon},u_{\mu,\varepsilon},w_{\mu,\varepsilon})$ be weak solutions to (\ref{b40}) satisfying (\ref{c13}) and (\ref{c14}), then the following estimates holds
\begin{equation}\label{c15}
\begin{aligned}
&\|\partial_{t}\sqrt{\xi_{\mu,\varepsilon}}\|_{L^{\infty}(0,T;W^{-1,\frac{3}{2}})}+\|\sqrt{\xi_{\mu,\varepsilon}}
\|_{L^{2}(0,T;H^{2})}\leq K,\\
&\parallel\xi_{\mu,\varepsilon}\parallel_{ L^{2}(0,T;H^{6})}+ \parallel\partial_{t}\xi_{\mu,\varepsilon}\parallel_{ L^{2}(0,T;H^{-1})}\leq K,\\
&\parallel\xi_{\mu,\varepsilon}u_{\mu,\varepsilon}\parallel_{ L^{2}(0,T;W^{1,\frac{3}{2}})}+ \parallel\partial_{t}(\xi_{\mu,\varepsilon}u_{\mu,\varepsilon})\parallel_{L^{2}(0,T;H^{-5})}\leq K,\\
& \parallel\xi_{\mu,\varepsilon}^{-10}\parallel_{ L^{\frac{5}{3}}(0,T;L^{\frac{5}{3}})}\leq K,\\
\end{aligned}
\end{equation}
where $K$ is independent of $ \varepsilon,\mu$.
\begin{proof}
The proof is similar to the compactness analysis in Section 2, we can prove the above resluts. For the simplicity, we omit the details here.
\end{proof}

With Lemma \ref{le2} and Lemma \ref{l4} in hand, we have the following compactness results as $\mu\rightarrow0=\varepsilon\rightarrow0$:
\begin{equation}\label{c16}
\begin{aligned}
&\sqrt{\xi_{\mu,\varepsilon}}\rightarrow \sqrt{\xi},\  strongly\ in\  L^{2}([0,T];H^{1}), \\
&\xi_{\mu,\varepsilon}\rightarrow \xi,\ strongly\ in\  C([0,T];H^{5}),\ \xi_{\mu,\varepsilon}\rightharpoonup \xi,\ in\ L^{2}(0,T;H^{6}),\ \\
&u_{\mu,\varepsilon}\rightharpoonup u,\ in\ L^{2}(0,T;L^{2}), \ \sqrt{\xi_{\mu,\varepsilon}}\rightharpoonup \sqrt{\xi},\ in\ L^{2}(0,T;H^{2}),\\
&\xi_{\mu,\varepsilon}u_{\mu,\varepsilon}\rightarrow \xi u\ strongly\ in\ L^{2}(0,T;L^{p}),\ for\ \forall \ 1\leq p<3,\\
&\xi_{\mu,\varepsilon}^{-10}\rightarrow \xi^{-10},\ strongly\ in\  L^{1}(0,T;L^{1}),\\
&\sqrt{\xi_{\mu,\varepsilon}}u_{\mu,\varepsilon}\rightarrow\sqrt{\xi}u\ \ strongly\ in\ L^{2}(0,T;L^{2}).
\end{aligned}
\end{equation}\end{Lemma}
Hence, $\sqrt{\xi_{\mu,\varepsilon}}\rightarrow\sqrt{\xi},\ \xi_{\mu,\varepsilon}\rightarrow\xi \ and\ \xi_{\mu,\varepsilon}^{-10}\rightarrow \xi^{-10}$ almost everywhere.

Thanks to $\sqrt{\xi_{\mu,\varepsilon}}\partial_{z}w_{\mu,\varepsilon}\in L^{2}(0,T;L^{2})$, by the Poincar\'{e} inequality, we have
$$
\int_{0}^{h}|\sqrt{\xi_{\mu,\varepsilon}}w_{\mu,\varepsilon}|^{2}dz\leq c\int_{0}^{h}|\partial_{z}(\sqrt{\xi_{\mu,\varepsilon}}w_{\mu,\varepsilon})|^{2}dz.
$$
Thus the estimates
$$
\int_{0}^{T}\int_{\Omega}|\sqrt{\xi_{\mu,\varepsilon}}w_{\mu,\varepsilon}|^{2}dxdzdt\leq c\int_{0}^{T}\int_{\Omega}\xi_{\mu,\varepsilon}|\partial_{z}w_{\mu,\varepsilon}|^{2}dxdzdt
$$
gives $\sqrt{\xi_{\mu,\varepsilon}}w_{\mu,\varepsilon}\in L^{2}(0,T;L^{2}) $. Consequently, there exists,
up to a subsequence , $\sqrt{\xi_{\mu,\varepsilon}}w_{\mu,\varepsilon}\rightharpoonup l\in L^{2}(0,T;L^{2})$ weakly. Since $\sqrt{\xi_{\mu,\varepsilon}}\rightarrow \sqrt{\xi}$ strongly, we define $w$ as following:
\begin{equation}\label{c17}
w=\begin{cases}
\frac{l}{\sqrt{\xi}}\ \ \ \ \  \ \ \ if\ \ \xi>0,\\
0 \ \ \ a.e.\ \ if \ \ \xi=0.
\end{cases}
\end{equation}
Then the limit $l=\sqrt{\xi}(\dfrac{l}{\sqrt{\xi}})=\sqrt{\xi}w$.

By the above regularity results, we are ready to pass to the limits as $\mu=\varepsilon\rightarrow0$.
For any test function $\varphi\in C^{\infty}([0,T];\Omega)$, when $\mu=\varepsilon\rightarrow0$, we have
\begin{equation}\label{c18}
\begin{aligned}
\int_{0}^{T}\int_{\Omega}\partial_{z}(\xi_{\mu,\varepsilon}u_{\mu,\varepsilon}w_{\mu,\varepsilon})\varphi dxdzdt&=-\int_{0}^{T}\int_{\Omega}\xi_{\mu,\varepsilon}u_{\mu,\varepsilon}w_{\mu,\varepsilon}\partial_{z}\varphi dxdzdt\\
&\rightarrow-\int_{0}^{T}\int_{\Omega}\xi uw\partial_{z}\varphi dxdzdt,
\end{aligned}
\end{equation}
and
\begin{equation}\label{c19}
\begin{aligned}
\int_{0}^{T}\int_{\Omega}\partial_{z}(\xi_{\mu,\varepsilon}w_{\mu,\varepsilon})\varphi dxdzdt&=-\int_{0}^{T}\int_{\Omega}\xi_{\mu,\varepsilon}w_{\mu,\varepsilon}\partial_{z}\varphi dxdzdt\\
&\rightarrow-\int_{0}^{T}\int_{\Omega}\xi w\partial_{z}\varphi dxdzdt.
\end{aligned}
\end{equation}
Moreover,
\begin{equation}\label{c20}
\begin{aligned}
&\varepsilon\int_{0}^{T}\int_{\Omega}\nabla_{x}\xi_{\mu,\varepsilon}\cdot\nabla_{x} u_{\mu,\varepsilon}\varphi dxdzdt\\
&=-\varepsilon\int_{0}^{T}\int_{\Omega}\Delta_{x}\xi_{\mu,\varepsilon} u_{\mu,\varepsilon}\varphi dxdzdt-\varepsilon\int_{0}^{T}\int_{\Omega}\nabla_{x}\xi_{\mu,\varepsilon}u_{\mu,\varepsilon}{\rm div}_{x}\varphi dxdzdt\\
&\leq \varepsilon\|\Delta_{x}\xi_{\mu,\varepsilon}\|_{L^{2}(L^{2})}\| u_{\mu,\varepsilon}\|
_{L^{2}(L^{2})}\|\varphi\|_{L^{\infty}(L^{\infty})}\\
&\ \ \ +\varepsilon\|\nabla_{x}\xi_{\mu,\varepsilon}\|_{L^{2}(L^{2})}\| u_{\mu,\varepsilon}\|
_{L^{2}(L^{2})}\|{\rm div}_{x}\varphi\|_{L^{\infty}(L^{\infty})}\rightarrow0,\ \ as\ \ \varepsilon\rightarrow0,
\end{aligned}
\end{equation}
and
\begin{equation}\label{c21}
\begin{aligned}
\mu\int_{0}^{T}\int_{\Omega}\Delta^{2}u_{\mu,\varepsilon}\varphi dxdzdt
&=\mu\int_{0}^{T}\int_{\Omega}\Delta
u_{\mu,\varepsilon}\Delta\varphi dxdzdt\\
&\leq \sqrt{\mu}\|\sqrt{\mu}\Delta u_{\mu,\varepsilon}\|_{L^{2}(L^{2})}\|\Delta \varphi\|_{L^{2}(L^{2})}\rightarrow0,\ as\ \mu\rightarrow0.
\end{aligned}
\end{equation}

So by taking $\mu=\varepsilon\rightarrow0$ in $(\ref{b40})$, we have
\begin{equation}\label{c22}
\left\{\begin{array}{lll}
\partial_{t}\xi+{\rm div}_{x}(\xi u)+\partial_{z}(\xi w)=0,\\
\partial_{t}(\xi u)+{\rm div}_{x}(\xi u\otimes u)+\partial_{z}(\xi uw)+\nabla_{x}\xi+r_{0} u+r\xi |u|u\\
=2\overline{\nu}_{1}{\rm div}_{x}(\xi D_{x}(u))+\overline{\nu}_{2}\partial_{z}(\xi \partial_{z}u)+
\eta\nabla_{x}\xi^{-10}+\kappa\xi\nabla_{x}(\dfrac{\Delta_{x}\sqrt{\xi}}{\sqrt{\xi}})
+\delta\xi\nabla_{x}\Delta_{x}^{5}\xi,\\
\partial_{z}\xi=0,
\end{array}\right.
\end{equation}
holds in the sense of distribution on $(0; T)\times\Omega$.

Furthermore, thanks to lower semi-continuity of the convex function and the strong convergence of $(\xi_{\mu,\varepsilon},u_{\mu,\varepsilon},w_{\mu,\varepsilon})$, we can pass to the limits in the energy inequality (\ref{b39}) and B-D entropy (\ref{c3}) as $\mu=\varepsilon\rightarrow0$ with $\delta,\eta,r_{0},\kappa$ being fixed. And we have
\begin{equation}\label{c23}
\begin{aligned}
&\int_{\Omega}\frac{1}{2}\xi u^{2}+\xi\ln\xi-\xi+1+\frac{1}{11}\eta\xi^{-10}+\kappa|\nabla_{x}\sqrt{\xi}|^{2}+
\frac{\delta}{2}|\nabla_{x}\Delta_{x}^{2}\xi|^{2}dxdz\\
&+r_{0}\int_{0}^{T}\int_{\Omega}u^{2}dxdzdt+r\int_{0}^{T}\int_{\Omega}\xi|u|^{3}dxdzdt\\
&+\int_{0}^{T}\int_{\Omega} 2\overline{\nu}_{1}\xi|D_{x}(u)|^{2}dxdzdt
+\int_{0}^{T}\int_{\Omega}\overline{\nu}_{2}\xi|\partial_{z}u|^{2}dxdzdt\\
&\leq \int_{\Omega}\frac{1}{2}\xi_{0} u_{0}^{2}+\xi_{0}\ln\xi_{0}-\xi_{0}+1+\frac{1}{11}\eta\xi_{0}^{-10}+\kappa|\nabla_{x}\sqrt{\xi_{0}}|^{2}+
\frac{\delta}{2}|\nabla_{x}\Delta_{x}^{2}\xi_{0}|^{2}dxdz,
\end{aligned}
\end{equation}
and
\begin{equation*}
\begin{aligned}
&\int_{\Omega}(\frac{1}{2}\xi(u+2\overline{\nu}_{1} \nabla_{x}\ln \xi)^{2}-2\overline{\nu}_{1}r_{0}\ln\xi) dxdzdt
+2\overline{\nu}_{1}\int_{0}^{T}\int_{\Omega}\xi|\partial_{z}w|^{2}dxdzdt\\
&+8\overline{\nu}_{1}\int_{0}^{T}\int_{\Omega}
|\nabla_{x}\sqrt{\xi}|^{2}dxdzdt +r\int_{0}^{T}\int_{\Omega}\xi|u|^{3}dxdzdt+\overline{\nu}_{2}\int_{0}^{T}\int_{\Omega}\xi|\partial_{z}u|^{2}dxdzdt\\
&+\frac{8}{5}\eta\overline{\nu}_{1}\int_{0}^{T}\int_{\Omega}|\nabla_{x}\xi^{-5}|^{2}dxdzdt
+\kappa\overline{\nu}_{1}\int_{0}^{T}
\int_{\Omega}\xi|\nabla_{x}^{2}\ln\xi|^{2}dxdzdt\\
\end{aligned}
\end{equation*}
\begin{equation}\label{c24}
\begin{aligned}
&+r_{0}\int_{0}^{T}\int_{\Omega}u^{2}dxdzdt
+2\delta\overline{\nu}_{1}\int_{0}^{T}\int_{\Omega}|\Delta_{x}^{3}\xi|^{2}dxdzdt+\overline{\nu}_{1}
\int_{0}^{T}\int_{\Omega}\xi|A_{x}(u)|^{2}dxdzdt\\
&\leq\int_{\Omega}(\frac{1}{2}\xi_{0}(u_{0}+2\overline{\nu}_{1} \nabla_{x}\ln \xi_{0})^{2}
-2\overline{\nu}_{1}r_{0}\ln\xi_{0}) dxdz+E_{0}+C,
\end{aligned}
\end{equation}

Thus, to conclude this part, we have got that the following proposition
\begin{Proposition}\label{pro4}
There exists the weak solutions to the system
\begin{equation}\label{c25}
\left\{\begin{array}{lll}
\partial_{t}\xi+{\rm div}_{x}(\xi u)+\partial_{z}(\xi w)=0,\\
\partial_{t}(\xi u)+{\rm div}_{x}(\xi u\otimes u)+\partial_{z}(\xi uw)+\nabla_{x}\xi+r_{0} u+r\xi |u|u=\\
2\overline{\nu}_{1}{\rm div}_{x}(\xi D_{x}(u))+\overline{\nu}_{2}\partial_{z}(\xi \partial_{z}u)+
\eta\nabla_{x}\xi^{-10}+\kappa\xi\nabla_{x}(\dfrac{\Delta_{x}\sqrt{\xi}}{\sqrt{\xi}})
+\delta\xi\nabla_{x}\Delta_{x}^{5}\xi,\\
\partial_{z}\xi=0.
\end{array}\right.
\end{equation}
with suitable initial data, for any $T>0$. In particular, the weak solutions $(\xi,u,w)$ satisfy energy inequality (\ref{c23}) and the B-D entropy (\ref{c24}).
\end{Proposition}
\section{Passing to the limits as $\eta\rightarrow0$.}
In this section, we pass to the limits as $\eta\rightarrow0$ with $\delta,\kappa,r_{0}$ being fixed. We denote by $(\xi_{\eta},u_{\eta},w_{\eta})$ the weak solutions at this level. From Proposition \ref{pro4}, we have the following regularities
\begin{equation}\label{d1}
\begin{aligned}
&\sqrt{\xi_{\eta}}u_{\eta}\in L^{\infty}(0,T;L^{2}),\ \sqrt{\xi_{\eta}}D_{x}(u_{\eta})\in L^{2}(0,T;L^{2}),
\sqrt{\xi_{\eta}}\partial_{z}u_{\eta}\in L^{2}(0,T;L^{2})\\
&\nabla_{x}\sqrt{\xi_{\eta}}\in L^{\infty}(0,T;L^{2}),\ \sqrt{\delta}\xi_{\eta}\in L^{\infty}(0,T;H^{5}),\ \sqrt{\delta}\xi_{\eta}\in L^{2}(0,T;H^{6}),\\
&u_{\xi_{\eta}}\in L^{2}(0,T;L^{2}),\ \xi_{\eta}^{\frac{1}{3}}u_{\eta}\in L^{3}(0,T;L^{3}),\ \sqrt{\xi_{\eta}}\nabla u_{\eta}\in L^{2}(0,T;L^{2}),\\
&\xi_{\eta}\ln\xi_{\eta}-\xi_{\eta}+1\in L^{\infty}(0,T;L^{1}),
-r_{0}\ln\xi_{\eta}\in L^{\infty}(0,T;L^{1}),\\
&\eta\xi_{\eta}^{-10}\in L^{\infty}(0,T;L^{1}),\ \sqrt{\eta}\nabla_{x} \xi_{\eta}^{-5}\in L^{2}(0,T;L^{2}),\sqrt{\xi_{\eta}}\partial_{z}w_{\eta}\in L^{2}(0,T;L^{2}),\\
&\sqrt{\kappa}\sqrt{\xi_{\eta}}\in L^{2}(0,T;H^{2}),
 \sqrt[4]{\kappa}\nabla_{x}\xi_{\eta}^{\frac{1}{4}}\in L^{4}(0,T;L^{4}).
\end{aligned}
\end{equation}

So it is easy to check that we have the similar estimates as in Lemma \ref{l4} uniform with $\eta$. Furthermore, we deduce the similar compactness for $(\xi_{\eta},u_{\eta},w_{\eta})$ as follows
\begin{equation}\label{d2}
\begin{aligned}
&\sqrt{\xi_{\eta}}\rightarrow \sqrt{\xi},\ strongly\ in\  L^{2}([0,T];H^{1}),\\
&\xi_{\eta}\rightarrow \xi \ strongly\ in\  C([0,T];H^{5}),\ \xi_{\eta}\rightharpoonup \xi,\ in\ L^{2}(0,T;H^{6}),\ \\
&u_{\eta}\rightharpoonup u\ in\ L^{2}(0,T;L^{2}),\ \xi_{\eta}w_{\eta}\rightharpoonup \xi w\ in\  L^{2}(0,T;L^{2}),\\
&\xi_{\eta}u_{\eta}\rightarrow \xi u\ strongly\ in\ L^{2}(0,T;L^{p}),\ for\ \forall 1\leq p<3,\\
&\sqrt{\xi_{\eta}}u_{\eta}\rightarrow\sqrt{\xi}u\ \ strongly\ in\ L^{2}(0,T;L^{2}),\ \sqrt{\xi_{\eta}}\rightharpoonup \sqrt{\xi},\ in\ L^{2}(0,T;H^{2}).
\end{aligned}
\end{equation}
Therefore, at this level of approximation, we only focus on the convergence of the cold pressure $\eta\nabla\xi_{\eta}^{-10}$. Here we state the following lemma.
 \begin{Lemma}\label{l5}
 For $\xi_{\eta}$ defined as in Proposition \ref{pro4}, we have
 $$\eta\int_{0}^{T}\int_{\Omega} \xi_{\eta}^{-10}dxdzdt\rightarrow0$$
 as $\eta\rightarrow0$.
 \end{Lemma}
 \begin{proof}
 The proof is motivated by the method in \cite{yucheng2016}. From the B-D entropy (\ref{c24}), we have
 \begin{equation}\label{d3}
 \sup_{t\in[0,T]}\int_{\Omega} (\log(\frac{1}{\xi_{\eta}}))_{+}dxdz\leq C(r_{0})<+\infty.
 \end{equation}
 Note that
 $$y\in \mathbb{R}^{+}\rightarrow\log(\frac{1}{y})_{+}$$
 is a convex continuous function. Moreover, in combination with the property of the convex function and Fatou's Lemma, it yields that
 \begin{equation}\label{d4}
 \begin{aligned}
 \int_{\Omega}(\log(\frac{1}{\xi}))_{+}dxdz&\leq \int_{\Omega} \lim_{\eta\rightarrow0} \inf(\log(\frac{1}{\xi_{\eta}}))_{+}dxdz\\
 &\leq \lim_{\eta\rightarrow0}\inf \int_{\Omega} (\log(\frac{1}{\xi_{\eta}}))_{+}dxdz,
 \end{aligned}
 \end{equation}
 which implies $(\log(\frac{1}{\xi}))_{+}$ is bounded in $L^{\infty}(0,T;L^{1})$. So it allows us to deduce that
 \begin{equation}\label{d5}
 |\{x\ |\ \xi(t,x)=0\}|=0,\ \ for\ almost\ every\ t\in [0,T],
 \end{equation}
 where $|A|$ denotes the measure of set $A$.

Due to $\xi_{\eta}\rightarrow\xi$ strongly in $C([0,T];H^{5})$, hence $\xi_{\eta}\rightarrow\xi$ a.e.. Then the above limit and (\ref{d5}) deduce
 \begin{equation}\label{d6}
 \eta\xi_{\eta}^{-10}\rightarrow0\ \ a.e.
  \ \ as\  \eta \rightarrow 0.
 \end{equation}

 Moreover,using the interpolation inequality, we have
 $$\|\eta\xi_{\eta}^{-10}\|_{L^{\frac{5}{3}}(0,T;L^{\frac{5}{3}})}\leq \|\eta\xi_{\eta}^{-10}\|_{L^{\infty}(0,T;L^{1})}^{\frac{2}{5}}
 \|\eta\xi_{\eta}^{-10}\|_{L^{1}(0,T;L^{3})}^{\frac{3}{5}}\leq C,$$
 which combine with (\ref{d6}) and Lemma \ref{le3}, we have
 $$\eta\xi_{\eta}^{-10}\rightarrow0,\ strongly\ in\ L^{1}(0,T;L^{1}).$$
 \end{proof}

 Thus, by the compactness results (\ref{d2}), we can pass to the limit as $\eta\rightarrow0$ in (\ref{c25})
 \begin{equation}\label{d7}
\left\{\begin{array}{lll}
\partial_{t}\xi+{\rm div}_{x}(\xi u)+\partial_{z}(\xi w)=0,\\
\partial_{t}(\xi u)+{\rm div}_{x}(\xi u\otimes u)+\partial_{z}(\xi uw)+\nabla_{x}\xi+r_{0} u+r\xi |u|u\\
=2\overline{\nu}_{1}{\rm div}_{x}(\xi D_{x}(u))+\overline{\nu}_{2}\partial_{z}(\xi \partial_{z}u)+
\kappa\xi\nabla_{x}(\dfrac{\Delta_{x}\sqrt{\xi}}{\sqrt{\xi}})
+\delta\xi\nabla_{x}\Delta_{x}^{5}\xi,\\
\partial_{z}\xi=0
\end{array}\right.
\end{equation}
holds in the sense of distribution on $(0; T)\times\Omega$.

 Due to the lower semi-continuity of convex functions, we can obtain the following energy inequality and B-D entropy by passing to the limits in (\ref{c23}) and (\ref{c24}) as $\eta\rightarrow0$.
 \begin{equation}\label{d8}
\begin{aligned}
&\int_{\Omega}\frac{1}{2}\xi u^{2}+\xi\ln\xi-\xi+1+\kappa|\nabla_{x}\sqrt{\xi}|^{2}+
\frac{\delta}{2}|\nabla_{x}\Delta_{x}^{2}\xi|^{2}dxdz\\
&+r_{0}\int_{0}^{T}\int_{\Omega}u^{2}dxdzdt+r\int_{0}^{T}\int_{\Omega}\xi|u|^{3}dxdzdt\\
&+\int_{0}^{T}\int_{\Omega} 2\overline{\nu}_{1}\xi|D_{x}(u)|^{2}dxdzdt
+\int_{0}^{T}\int_{\Omega}\overline{\nu}_{2}\xi|\partial_{z}u|^{2}dxdzdt\\
&\leq \int_{\Omega}\frac{1}{2}\xi_{0} u_{0}^{2}+\xi_{0}\ln\xi_{0}-\xi_{0}+1+\kappa|\nabla_{x}\sqrt{\xi_{0}}|^{2}+
\frac{\delta}{2}|\nabla_{x}\Delta_{x}^{2}\xi_{0}|^{2}dxdz,
\end{aligned}
\end{equation}
and
\begin{equation}\label{d9}
\begin{aligned}
&\int_{\Omega}(\frac{1}{2}\xi(u+2\overline{\nu}_{1} \nabla_{x}\ln \xi)^{2}-2\overline{\nu}_{1}r_{0}\ln\xi) dxdzdt
+2\overline{\nu}_{1}\int_{0}^{T}\int_{\Omega}\xi|\partial_{z}w|^{2}dxdzdt\\
&+8\overline{\nu}_{1}\int_{0}^{T}\int_{\Omega}
|\nabla_{x}\sqrt{\xi}|^{2}dxdzdt +r\int_{0}^{T}\int_{\Omega}\xi|u|^{3}dxdzdt\\
&+\overline{\nu}_{2}\int_{0}^{T}\int_{\Omega}\xi|\partial_{z}u|^{2}dxdzdt
+\kappa\overline{\nu}_{1}\int_{0}^{T}
\int_{\Omega}\xi|\nabla_{x}^{2}\ln\xi|^{2}dxdzdt\\
&+r_{0}\int_{0}^{T}\int_{\Omega}u^{2}dxdzdt
+2\delta\overline{\nu}_{1}\int_{0}^{T}\int_{\Omega}|\Delta_{x}^{3}\xi|^{2}dxdzdt\\
&+\overline{\nu}_{1}\int_{0}^{T}\int_{\Omega}\xi|A_{x}(u)|^{2}dxdzdt\\
&\leq\int_{\Omega}(\frac{1}{2}\xi_{0}(u_{0}+2\overline{\nu}_{1} \nabla_{x}\ln \xi_{0})^{2}
-2\overline{\nu}_{1}r_{0}\ln\xi_{0}) dxdz+E_{0}+C.
\end{aligned}
\end{equation}
Thus we have the following proposition on the existence of the weak solutions at this level of approximation.
\begin{Proposition}\label{pro5}
For any $T>0$, there exist weak solutions to the system (\ref{d7}) with suitable initial data. In particular, the weak solutions $(\xi,u,w)$ satisfy the energy inequality (\ref{d8}) and the B-D entropy (\ref{d9}).
\end{Proposition}
\section{Passing to the limits as $\kappa,\delta,r_{0}\rightarrow0$.}
At this level, the weak solutions satisfy the energy inequality (\ref{d8}) and the B-D entropy (\ref{d9}). Thus we have the following regularities:
\begin{equation*}
\begin{aligned}
&\sqrt{\xi_{\delta,\kappa,r_{0}}}u_{\delta,\kappa,r_{0}}\in L^{\infty}(0,T;L^{2}),\ \sqrt{\xi_{\delta,\kappa,r_{0}}}D_{x}(u_{\delta,\kappa,r_{0}})\in L^{2}(0,T;L^{2}),\\
&\nabla_{x}\sqrt{\xi_{\delta,\kappa,r_{0}}}\in L^{\infty}(0,T;L^{2}),
\sqrt{\xi_{\delta,\kappa,r_{0}}}\partial_{z}u_{\delta,\kappa,r_{0}}\in L^{2}(0,T;L^{2}),\\ &\sqrt{\delta}\xi_{\delta,\kappa,r_{0}}\in L^{\infty}(0,T;H^{5})\cap L^{2}(0,T;H^{6}),
u_{\delta,\kappa,r_{0}}\in L^{2}(0,T;L^{2}),\\
&\xi_{\delta,\kappa,r_{0}}^{\frac{1}{3}}u_{\delta,\kappa,r_{0}}\in L^{3}(0,T;L^{3}),\ \sqrt{\xi_{\delta,\kappa,r_{0}}}\nabla u_{\delta,\kappa,r_{0}}\in L^{2}(0,T;L^{2}),\\
\end{aligned}
\end{equation*}
\begin{equation}\label{e1}
\begin{aligned}
&\xi_{\delta,\kappa,r_{0}}\ln\xi_{\delta,\kappa,r_{0}}-\xi_{\delta,\kappa,r_{0}}+1\in L^{\infty}(0,T;L^{1}),
-r_{0}\ln\xi_{\delta,\kappa,r_{0}}\in L^{\infty}(0,T;L^{1}),\\
&\sqrt{\xi_{\delta,\kappa,r_{0}}}\partial_{z}w_{\delta,\kappa,r_{0}}\in L^{2}(0,T;L^{2}),\sqrt{\kappa}\sqrt{\xi_{\delta,\kappa,r_{0}}}\in L^{2}(0,T;H^{2}),\\
& \sqrt[4]{\kappa}\nabla_{x}\xi_{\delta,\kappa,r_{0}}^{\frac{1}{4}}\in L^{4}(0,T;L^{4}),
\end{aligned}
\end{equation}

Next, we will proceed the compactness arguments in several steps
\subsection{ Convergence of $\sqrt{\xi_{\delta,\kappa,r_{0}}}$.}
\begin{Lemma}\label{l6}
For $\xi_{\delta,\kappa,r_{0}}$ satisfy  Proposition \ref{pro5}, we have
$$\sqrt{\xi_{\delta,\kappa,r_{0}}} is\ bounded\  in\  L^{\infty}(0,T;H^{1}),$$
$$\partial_{t}\sqrt{\xi_{\delta,\kappa,r_{0}}} is\ bounded\  in\  L^{2}(0,T;H^{-1}).$$
Then, up to a subsequence,we get $\sqrt{\xi_{\delta,\kappa,r_{0}}}$ convergence almost everywhere and convergence strongly  in $L^{2}(0,T;L^{2})$, which means
$$\sqrt{\xi_{\delta,\kappa,r_{0}}}\rightarrow \sqrt{\xi},\ \ a.e.\ and\ strongly\ in\ L^{2}(0,T;L^{2}).$$
Moreover, we have
$$\xi_{\delta,\kappa,r_{0}}\rightarrow \xi\ \ a.e.\ and\ strongly\ in\ C([0,T];L^{p}),\ \ for \ any\ p\in [1,3).$$

\end{Lemma}
\begin{proof}
Since$\|\sqrt{\xi_{\delta,\kappa,r_{0}}}\|_{L^{2}}^{2}=\|\xi_{0}\|_{L^{1}}$ and  $\nabla_{x}\sqrt{\xi_{\delta,\kappa,r_{0}}}$$\in L^{\infty}(0,T;L^{2})$, we have $\sqrt{\xi_{\delta,\kappa,r_{0}}} \in L^{\infty}(0,T;H^{1})$. Next
\begin{equation}\label{e2}
\begin{aligned}
\partial_{t}\sqrt{\xi_{\delta,\kappa,r_{0}}}&=-\frac{1}{2}\sqrt{\xi_{\delta,\kappa,r_{0}}}{\rm div}_{x}(u_{\delta,\kappa,r_{0}})-u_{\delta,\kappa,r_{0}}\cdot\nabla_{x}\sqrt{\xi_{\delta,\kappa,r_{0}}}-\frac{1}{2}
\sqrt{\xi_{\delta,\kappa,r_{0}}}\partial_{z}w_{\delta,\kappa,r_{0}}\\
&=\frac{1}{2}\sqrt{\xi_{\delta,\kappa,r_{0}}}{\rm div}_{x}(u_{\delta,\kappa,r_{0}})-{\rm div}_{x}(\sqrt{\xi_{\delta,\kappa,r_{0}}}u_{\delta,\kappa,r_{0}})-\frac{1}{2}
\sqrt{\xi_{\delta,\kappa,r_{0}}}\partial_{z}w_{\delta,\kappa,r_{0}},
\end{aligned}
\end{equation}
which yields $\partial_{t}\sqrt{\xi_{\delta,\kappa,r_{0}}}\in L^{2}(0,T;H^{-1})$. Using Lemma \ref{le2}, we have
$$\sqrt{\xi_{\delta,\kappa,r_{0}}}\rightarrow \sqrt{\xi},\ \  strongly\ in\ L^{2}(0,T;L^{2}),$$
and it yields that $\sqrt{\xi_{\delta,\kappa,r_{0}}}\rightarrow \sqrt{\xi}\ a.e.$.
By using the Sobolev embedding theorem, we have
$\sqrt{\xi_{\delta,\kappa,r_{0}}} $ is bounded in $ L^{\infty}(0,T;L^{6})$,\
then$$\xi_{\delta,\kappa,r_{0}}\in L^{\infty}(0,T;L^{3}).$$
And we deduce
\begin{equation}\label{e3}
\begin{aligned}
\xi_{\delta,\kappa,r_{0}}u_{\delta,\kappa,r_{0}}=\sqrt{\xi_{\delta,\kappa,r_{0}}}\sqrt{\xi_{\delta,\kappa,r_{0}}}
u_{\delta,\kappa,r_{0}}\in L^{\infty}(0,T;L^{\frac{3}{2}}),
\end{aligned}
\end{equation}
which yields ${\rm div}_{x}(\xi_{\delta,\kappa,r_{0}}u_{\delta,\kappa,r_{0}})\in L^{\infty}(0,T;W^{-1,\frac{3}{2}})$. This combines with
$\sqrt{\xi_{\delta,\kappa,r_{0}}}$ $\partial_{z}w_{\delta,\kappa,r_{0}}$ and the continuity equation yield $\partial_{t}\xi_{\delta,\kappa,r_{0}}\in L^{2}(0,T;W^{-1,\frac{3}{2}})$.

Furthermore, we have
\begin{equation}\label{e4}
\begin{aligned}
\nabla_{x}{\xi_{\delta,\kappa,r_{0}}}=2\sqrt{\xi_{\delta,\kappa,r_{0}}}\nabla_{x}\sqrt{\xi_{\delta,\kappa,r_{0}}}\in L^{\infty}(0,T;L^{\frac{3}{2}}).
\end{aligned}
\end{equation}
 So $\xi_{\delta,\kappa,r_{0}}$ is bounded in $L^{\infty}(0,T;W^{1,\frac{3}{2}})$.
 We use Lemma \ref{le2} to get
\begin{equation}\label{e5}
\xi_{\delta,\kappa,r_{0}}\rightarrow \xi,\ strongly\ in\ C([0,T];L^{p}),\ for\ any\ p\in[1,3).
\end{equation}
And hence, we have
$$\xi_{\delta,\kappa,r_{0}}\rightarrow \xi\ a.e..$$
Thus the proof of this lemma is completed.
\end{proof}
\subsection{ Convergence of the momentum.}
\begin{Lemma}\label{l8}
Up to a subsequence, the momentum $m_{\delta,\kappa,r_{0}}=\xi_{\delta,\kappa,r_{0}}u_{\delta,\kappa,r_{0}}$
 satisfy
 $$m_{\delta,\kappa,r_{0}}\rightarrow m\ \ strongly \ \ in\ \ L^{2}(0,T;L^{q})\ \ for\ all\ \ q\in [1,1.5).$$ In particular
 $$m_{\delta,\kappa,r_{0}}\rightarrow m\ a.e. \ for\ (x,t)\in \Omega\times(0,T).$$

 Remark: we can define $u(x,z,t)=m(x,z,t)/\xi(x,t)$ outside the vacuum set
 $\{x$ $|\xi(x,t)=0\}$. Then we obtain
 $$\xi_{\delta,\kappa,r_{0}}u_{\delta,\kappa,r_{0}}\rightarrow \xi u\ \ strongly \ \ in\ \ L^{2}(0,T;L^{q})\ \ for\ all\ \ q\in [1,1.5).$$
\end{Lemma}
\begin{proof}
Since
\begin{equation}\label{e6}
\begin{aligned}
&\nabla_{x}(\xi_{\delta,\kappa,r_{0}}u_{\delta,\kappa,r_{0}})=\nabla_{x}{\xi_{\delta,\kappa,r_{0}}}\otimes u_{\delta,\kappa,r_{0}}+\xi_{\delta,\kappa,r_{0}}\nabla_{x}{u_{\delta,\kappa,r_{0}}}\\
&=2\nabla_{x}\sqrt
{\xi_{\delta,\kappa,r_{0}}}\otimes
\sqrt{\xi_{\delta,\kappa,r_{0}}}u_{\delta,\kappa,r_{0}}+\sqrt{\xi_{\delta,\kappa,r_{0}}}\sqrt{\xi_{\delta,\kappa,r_{0}}}
\nabla{u_{\delta,\kappa,r_{0}}}\in L^{2}(0,T;L^{1}),
\end{aligned}
\end{equation}
$$
\partial_{z}(\xi_{\delta,\kappa,r_{0}}u_{\delta,\kappa,r_{0}})=\sqrt{\xi_{\delta,\kappa,r_{0}}}\sqrt
{\xi_{\delta,\kappa,r_{0}}}\partial_{z}u_{\delta,\kappa,r_{0}}\in L^{2}(0,T;L^{\frac{3}{2}})
$$
and (\ref{e3}), we deduce
$$\xi_{\delta,\kappa,r_{0}}u_{\delta,\kappa,r_{0}}\in L^{2}(0,T;W^{1,1}).$$
Moreover, we can show
$$\partial_{t}(\xi_{\delta,\kappa,r_{0}}u_{\delta,\kappa,r_{0}})\ is\  bounded \ in\  L^{2}(0,T;H^{-s}),\ \ for\ some \ constant\ s>0.$$
In fact,
\begin{equation}\label{e7}
\begin{aligned}
&\partial_{t}(\xi_{\delta,\kappa,r_{0}} u_{\delta,\kappa,r_{0}})=2\overline{\nu}_{1}{\rm div}_{x}(\xi_{\delta,\kappa,r_{0}} D_{x}(u_{\delta,\kappa,r_{0}}))+\overline{\nu}_{2}\partial_{z}(\xi_{\delta,\kappa,r_{0}} \partial_{z}u_{\delta,\kappa,r_{0}})\\
&+\kappa\xi_{\delta,\kappa,r_{0}}\nabla_{x}(\dfrac{\Delta_{x}\sqrt{\xi_{\delta,\kappa,r_{0}}}}
{\sqrt{\xi_{\delta,\kappa,r_{0}}}})
+\delta\xi_{\delta,\kappa,r_{0}}\nabla_{x}\Delta_{x}^{5}\xi_{\delta,\kappa,r_{0}}-{\rm div}_{x}(\xi_{\delta,\kappa,r_{0}} u_{\delta,\kappa,r_{0}}\otimes u_{\delta,\kappa,r_{0}})\\
&-\partial_{z}(\xi_{\delta,\kappa,r_{0}} u_{\delta,\kappa,r_{0}}w_{\delta,\kappa,r_{0}})-\nabla_{x}\xi_{\delta,\kappa,r_{0}}+r_{0} u_{\delta,\kappa,r_{0}}-r\xi_{\delta,\kappa,r_{0}} |u_{\delta,\kappa,r_{0}}|u_{\delta,\kappa,r_{0}}.\\
\end{aligned}
\end{equation}

With the estimates (\ref{e1}) in hand, we deduce
$$
\xi_{\delta,\kappa,r_{0}}u_{\delta,\kappa,r_{0}}\otimes u_{\delta,\kappa,r_{0}}\in L^{\infty}(0,T;L^{1}),
\ \ \xi_{\delta,\kappa,r_{0}}u_{\delta,\kappa,r_{0}}w_{\delta,\kappa,r_{0}}\in L^{2}(0,T;L^{1})
$$
and
$$
\xi_{\delta,\kappa,r_{0}}\partial_{z}u_{\delta,\kappa,r_{0}}\in L^{2}(0,T;L^{\frac{3}{2}}),\ \
\xi_{\delta,\kappa,r_{0}}D_{x}(u_{\delta,\kappa,r_{0}})\in L^{2}(0,T;L^{\frac{3}{2}}).
$$
Particularly, thanks to the Sobolev imbedding theorem, we have
$$
{\rm div}_{x}(\xi_{\delta,\kappa,r_{0}}u_{\delta,\kappa,r_{0}}\otimes u_{\delta,\kappa,r_{0}})\in L^{\infty}(0,T;W^{-2,2}),
\ \partial_{z}(\xi_{\delta,\kappa,r_{0}}u_{\delta,\kappa,r_{0}}w_{\delta,\kappa,r_{0}})\in L^{2}(0,T;W^{-2,2})
$$
and
$$
\partial_{z}(\xi_{\delta,\kappa,r_{0}}\partial_{z}u_{\delta,\kappa,r_{0}})\in L^{2}(0,T;W^{-2,2}),\ \
{\rm div}_{x}(\xi_{\delta,\kappa,r_{0}}D_{x}(u_{\delta,\kappa,r_{0}}))\in L^{2}(0,T;W^{-2,2}).
$$

Moreover, thanks to $\sqrt{\delta}\xi_{\delta,\kappa,r_{0}} \in L^{\infty}(0,T;H^{5})\cap L^{2}(0,T;H^{6}) $
and $\sqrt{\kappa}\sqrt{\xi_{\delta,\kappa,r_{0}}}\in L^{2}(0,T;H^{2})$, we get
$$
\delta\xi_{\delta,\kappa,r_{0}}\nabla_{x}\Delta_{x}^{5}\ \in L^{2}(0,T;W^{-5,2})
$$
and
\begin{equation*}
\begin{aligned}
&\kappa\xi_{\delta,\kappa,r_{0}}\nabla_{x}(\frac{\Delta_{x}\sqrt{\xi_{\delta,\kappa,r_{0}}}}{\sqrt{\xi_{\delta,\kappa,r_{0}}}})=
\kappa\nabla_{x}(\sqrt{\xi_{\delta,\kappa,r_{0}}}\Delta_{x}\sqrt{\xi_{\delta,\kappa,r_{0}}})-2\kappa\Delta_{x}\sqrt{\xi_{\delta,\kappa,r_{0}}}
\nabla_{x}\sqrt{\xi_{\delta,\kappa,r_{0}}}\\
&\in L^{2}(0,T;W^{-3,2}).
\end{aligned}
\end{equation*}
Then we obtain
$$\partial_{t}(\xi_{\delta,\kappa,r_{0}}u_{\delta,\kappa,r_{0}})\ is\  bounded \ in\  L^{2}(0,T;H^{-5}).$$
Hence, using  Lemma \ref{le2}, Lemma \ref{l8} is proved.
\end{proof}

With Lemma \ref{l6} and Lemma \ref{l8} in hands , similar to the proof of Lemma \ref{l9}, we can deduce
\begin{equation}\label{e8}
\sqrt{\xi_{\delta,\kappa,r_{0}}}u_{\delta,\kappa,r_{0}}\rightarrow\sqrt{\xi}u\ \ strongly\ in\ L^{2}(0,T;L^{2}).
\end{equation}
as $\kappa=\delta=r_{0}\rightarrow0$ .
\subsection{ Convergence of the terms ${\rm div_{x}}(\xi_{\delta,\kappa,r_{0}}D_{x}(u_{\delta,\kappa,r_{0}}))$, $ \kappa\xi_{\delta,\kappa,r_{0}}\nabla_{x}(\dfrac{\Delta_{x}\sqrt{\xi_{\delta,\kappa,r_{0}}}}{\sqrt{\xi_{\delta,
\kappa,r_{0}}}}), $ \\$r_{0}u_{\delta,\kappa,r_{0}}$ and  $\xi_{\delta,\kappa,r_{0}}\nabla_{x}\Delta_{x}^{5}\xi_{\delta,\kappa,r_{0}}$}
 For any test function $\varphi\in C_{0}^{\infty}((0,T);\Omega)$, we have
\begin{equation*}
\begin{aligned}
r_{0}\int_{0}^{T}\int_{\Omega} u_{\delta,\kappa,r_{0}}\varphi dxdzdt\leq \sqrt{r_{0}}\|\sqrt{r_{0}}u_{\delta,\kappa,r_{0}}\|_{L^{2}(L^{2})}\|\varphi\|_{L^{2}(L^{2})}\rightarrow
0,\ \ as\ r_{0}\rightarrow0.
\end{aligned}
\end{equation*}

To deal with the diffusion term ${\rm div_{x}}(
\xi_{\delta,\kappa,r_{0}}D_{x}(u_{\delta,\kappa,r_{0}}))$. Recalling (\ref{b34}), we have
\begin{equation}\label{e9}
\begin{aligned}
&\int_{0}^{T}\int_{\Omega} {\rm div}_{x}(
\xi_{\delta,\kappa,r_{0}}D_{x}(u_{\delta,\kappa,r_{0}}))\varphi dxdzdt\\
&=\frac{1}{2}\int_{0}^{T}\int_{\Omega}(\xi_{\delta,\kappa,r_{0}}u_{\delta,\kappa,r_{0}}\cdot\Delta_{x}\varphi+
2\nabla_{x}\varphi\cdot\nabla_{x}\sqrt{\xi_{\delta,\kappa,r_{0}}}
\cdot\sqrt{\xi_{\delta,\kappa,r_{0}}}u_{\delta,\kappa,r_{0}})dxdzdt\\
&\ \ \ +\frac{1}{2}\int_{0}^{T}\int_{\Omega}(\xi_{\delta,\kappa,r_{0}}u_{\delta,\kappa,r_{0}}\cdot{\rm div}_{x}(\nabla_{x}^{t}\varphi)+2\nabla_{x}^{t}\varphi\cdot\nabla_{x}\sqrt{\xi_{\delta,\kappa,r_{0}}}
\cdot\sqrt{\xi_{\delta,\kappa,r_{0}}}u_{\delta,\kappa,r_{0}})dxdzdt.\\
\end{aligned}
\end{equation}
Since $\nabla_{x}\sqrt{\xi_{\delta,\kappa,r_{0}}}\in L^{\infty}([0,T];L^{2})$, the sequence $\nabla_{x}\sqrt{\xi_{\delta,\kappa,r_{0}}}$ weakly converges. By using Lemma \ref{l6} , Lemma \ref{l8} and $\sqrt{\xi_{\delta,\kappa,r_{0}}}u_{\delta,\kappa,r_{0}}\rightarrow \sqrt{\xi}u\ \ strongly\ \ in\ L^{2}(0,T;L^{2})$, we have
\begin{equation}\label{e10}
\begin{aligned}
&\frac{1}{2}\int_{0}^{T}\int_{\Omega}(\xi_{\delta,\kappa,r_{0}}u_{\delta,\kappa,r_{0}}\cdot\Delta_{x}\varphi+
2\nabla_{x}\varphi\cdot\nabla_{x}\sqrt{\xi_{\delta,\kappa,r_{0}}}
\cdot\sqrt{\xi_{\delta,\kappa,r_{0}}}u_{\delta,\kappa,r_{0}})dxdzdt\\
&+\frac{1}{2}\int_{0}^{T}\int_{\Omega}(\xi_{\delta,\kappa,r_{0}}u_{\delta,\kappa,r_{0}}\cdot{\rm div}_{x}(\nabla_{x}^{t}\varphi)+2\nabla_{x}^{t}\varphi\cdot\nabla_{x}\sqrt{\xi_{\delta,\kappa,r_{0}}}
\cdot\sqrt{\xi_{\delta,\kappa,r_{0}}}u_{\delta,\kappa,r_{0}})dxdzdt\\
&\rightarrow\frac{1}{2}\int_{0}^{T}\int_{\Omega}(\xi u\cdot\Delta_{x}\varphi+2\nabla_{x}\varphi\cdot\nabla_{x}\sqrt{\xi}
\cdot\sqrt{\xi}u)dxdzdt\\
&+\frac{1}{2}\int_{0}^{T}\int_{\Omega}(\xi u\cdot{\rm div}_{x}(\nabla_{x}^{t}\varphi)+2\nabla_{x}^{t}\varphi\cdot\nabla_{x}\sqrt{\xi}
\cdot\sqrt{\xi}u)dxdzdt,\\
\end{aligned}
\end{equation}
as $\delta=\kappa=r_{0}\rightarrow0$.
Hence
$$
\int_{0}^{T}\int_{\Omega} {\rm div}_{x}(
\xi_{\delta,\kappa,r_{0}}D_{x}(u_{\delta,\kappa,r_{0}}))\varphi dxdzdt\rightarrow\int_{0}^{T}\int_{\Omega} {\rm div}_{x}(\xi D_{x}(u))\varphi dxdzdt.
$$
For the quantum term, recalling (\ref{e1}), we have
\begin{equation*}
\begin{aligned}
&\kappa\int_{0}^{T}\int_{\Omega}\xi_{\delta,\kappa,r_{0}}\nabla_{x}(\dfrac{\Delta_{x}\sqrt{\xi_{\delta,\kappa,r_{0}
}}}{\sqrt{\xi_{\delta,\kappa,r_{0}}}})\varphi dxdzdt\\
&=-\kappa\int_{0}^{T}\int_{\Omega}\Delta_{x}\sqrt{\xi_{\delta,\kappa,r_{0}}}\nabla_{x}\sqrt{\xi_{\delta,\kappa,r_{0}}}\varphi dxdzdt\\
&\ \ \ -\kappa\int_{0}^{T}\int_{\Omega}\Delta_{x}\sqrt{\xi_{\delta,\kappa,r_{0}}}\sqrt{\xi_{\delta,\kappa,r_{0}}}{\rm div}_{x}\varphi dxdzdt\\
&\leq\sqrt{\kappa}\parallel\sqrt{\kappa}\Delta_{x}\sqrt{\xi_{\delta,\kappa,r_{0}}}\parallel_{ L^{2}(0,T;L^{2})}\parallel\nabla_{x}\sqrt{\xi_{\delta,\kappa,r_{0}}}\parallel_{L^{2}(0,T;L^{2})}
\parallel\varphi\parallel_{L^{\infty}(0,T;L^{\infty})}\\
&\ \ \ +\sqrt{\kappa}\parallel\sqrt{\kappa}\Delta_{x}\sqrt{\xi_{\delta,\kappa,r_{0}}}\parallel_{ L^{2}(0,T;L^{2})}\parallel\sqrt{\xi_{\delta,\kappa,r_{0}}}\parallel_{L^{2}(0,T;L^{2})}\parallel{\rm div}_{x}\varphi\parallel_{L^{\infty}(0,T;L^{\infty})}\\
&\rightarrow0
\end{aligned}
\end{equation*}

Finally, we prove the convergence of the high order term $\xi_{\delta,\kappa,r_{0}}\nabla_{x}\triangle_{x}^{5}\xi_{\delta,\kappa,r_{0}}$.
Using Lemma \ref{le1}, we can prove
$$\|\nabla_{x}^{5}\xi_{\delta,\kappa,r_{0}}\|_{L^{2}}\leq C \|\nabla_{x}^{6}\xi_{\delta,\kappa,r_{0}}\|_{L^{2}}^{\frac{9}{11}}\|\xi_{\delta,\kappa,r_{0}}\|
_{L^{3}}^{\frac{2}{11}}.$$
Thus,
$$
\int_{0}^{T}\delta(\int_{\Omega}|\nabla_{x}^{5}\xi_{\delta,\kappa,r_{0}}dxdz)^{\frac{11}{9}}dt\leq C
(\sup_{t\in(0,T)}\|\xi_{\delta,\kappa,r_{0}}\|_{L^{3}})^{\frac{4}{9}}\int_{0}^{T}\delta\int_{\Omega}|\nabla_{x}^{6}
\xi_{\delta,\kappa,r_{0}}|^{2}dxdzdt
$$
which implies $\delta^{\frac{9}{22}}\nabla_{x}^{5}\xi_{\delta,\kappa,r_{0}}\in L^{\frac{22}{9}}(0,T;L^{2})$.

For any test function $\varphi\in C_{0}^{\infty}([0,T];\Omega)$, we have
$$\delta\int_{0}^{T}\int_{\Omega} \xi_{\delta,\kappa,r_{0}}\nabla_{x}\Delta_{x}^{5}\xi_{\delta,\kappa,r_{0}} \varphi dxdzdt=-\delta\int_{0}^{T}\int_{\Omega}\Delta_{x}^{2}{\rm div}_{x}(\xi_{\delta,\kappa,r_{0}} \varphi)\Delta_{x}^{3}\xi_{\delta,\kappa,r_{0}} dxdzdt.$$
Now, we deal with the most difficult term
\begin{equation}\label{e11}
\begin{aligned}
&|\delta\int_{0}^{T}\int_{\Omega}\Delta_{x}^{2}(\nabla_{x}\xi_{\delta,\kappa,r_{0}})\Delta_{x}^{3}\xi_{\delta,\kappa,r_{0}}\varphi
dxdzdt|\\
&\leq C\delta^{\frac{1}{11}}\|\sqrt{\delta}\nabla_{x}^{6}\xi_{\delta,\kappa,r_{0}}\|_{L^{2}(L^{2})}\|\delta
^{\frac{9}{22}}\nabla_{x}^{5}\xi_{\delta,\kappa,r_{0}}\|_{L^{\frac{22}{9}}(L^{2})}\|\varphi\|_{L^{11}(L^{\infty})}
\rightarrow0,
\end{aligned}
\end{equation}
as $\delta\rightarrow0$.

Similar to the above arguments, we can control the other terms from
$$\delta\int_{0}^{T}\int_{\Omega} \xi_{\delta,\kappa,r_{0}}\nabla_{x}\Delta_{x}^{5}\xi_{\delta,\kappa,r_{0}} \varphi dxdzdt.$$
Thus, we have
$$\delta\int_{0}^{T}\int_{\Omega} \xi_{\delta,\kappa,r_{0}}\nabla_{x}\Delta_{x}^{5}\xi_{\delta,\kappa,r_{0}} \varphi dxdzdt\rightarrow0$$
as $\delta\rightarrow0$.

Passing to the limits in (\ref{d7}) as $\delta\rightarrow0,\kappa\rightarrow0 \ \ and\ \  r_{0}\rightarrow0$, we have
 \begin{equation}\label{e12}
\left\{\begin{array}{lll}
\partial_{t}\xi+{\rm div}_{x}(\xi u)+\partial_{z}(\xi w)=0,\\
\partial_{t}(\xi u)+{\rm div}_{x}(\xi u\otimes u)+\partial_{z}(\xi uw)+\nabla_{x}\xi+r\xi |u|u\\
=2\overline{\nu}_{1}{\rm div}_{x}(\xi D_{x}(u))+\overline{\nu}_{2}\partial_{z}(\xi \partial_{z}u),\\
\partial_{z}\xi=0.
\end{array}\right.
\end{equation}
holds in the sense of distribution on $(0; T)\times\Omega$.

 Furthermore,thanks to the lower semi-continuity of the convex functions ,we can obtain the following energy inequality and B-D entropy by passing to the limits as $r_{0}=\delta=\kappa\rightarrow0$
 \begin{equation}\label{e13}
\begin{aligned}
&\int_{\Omega}\frac{1}{2}\xi u^{2}+\xi\ln\xi-\xi+1dxdz
+r\int_{0}^{T}\int_{\Omega}\xi|u|^{3}dxdzdt\\
&+\int_{0}^{T}\int_{\Omega} 2\overline{\nu}_{1}\xi|D_{x}(u)|^{2}dxdzdt
+\int_{0}^{T}\int_{\Omega}\overline{\nu}_{2}\xi|\partial_{z}u|^{2}dxdzdt\\
&\leq \int_{\Omega}\frac{1}{2}\xi_{0} u_{0}^{2}+\xi_{0}\ln\xi_{0}-\xi_{0}+1dxdz
\end{aligned}
\end{equation}
and
\begin{equation}\label{e14}
\begin{aligned}
&\int_{\Omega}(\frac{1}{2}\xi(u+2\overline{\nu}_{1} \nabla_{x}\ln \xi)^{2}) dxdzdt
+2\overline{\nu}_{1}\int_{0}^{T}\int_{\Omega}\xi|\partial_{z}w|^{2}dxdzdt\\
&+8\overline{\nu}_{1}\int_{0}^{T}\int_{\Omega}
|\nabla_{x}\sqrt{\xi}|^{2}dxdzdt
+r\int_{0}^{T}\int_{\Omega}\xi|u|^{3}dxdzdt\\
&+\overline{\nu}_{2}\int_{0}^{T}\int_{\Omega}\xi|\partial_{z}u|^{2}dxdzdt
+\overline{\nu}_{1}\int_{0}^{T}\int_{\Omega}\xi|A_{x}(u)|^{2}dxdzdt\\
&\leq\int_{\Omega}(\frac{1}{2}\xi_{0}(u_{0}+2\overline{\nu}_{1} \nabla_{x}\ln \xi_{0})^{2}
) dxdz+E_{0}+C.
\end{aligned}
\end{equation}
Thus we have completed the proof of Theorem \ref{theorem1}.
\subsection{ Proof of Theorem1.2}From Theorem \ref{theorem1}, we consider the weak solution $(\xi,u,w)$  of system (\ref{a3}). At this level, the weak solutions satisfy the energy inequality (\ref{e13}) and the B-D entropy (\ref{e14}), thus we have the following regularities:
\begin{equation}\label{e15}
\begin{aligned}
&\sqrt{\xi}u\in L^{\infty}(0,T;L^{2}),\ \nabla_{x}\sqrt{\xi}\in L^{\infty}(0,T;L^{2})\cap L^{2}(0,T;L^{2}),\\
&\sqrt{\xi}\partial_{z}u\in L^{2}(0,T;L^{2}), \xi^{\frac{1}{3}}u\in L^{3}(0,T;L^{3}),\ \sqrt{\xi}\nabla u\in L^{2}(0,T;L^{2}),\\
&\sqrt{\xi}\partial_{z}w\in L^{2}(0,T;L^{2}), \sqrt{\xi}D_{x}(u)\in L^{2}(0,T;L^{2}),
\end{aligned}
\end{equation}

In order to prove Theorem\ref{theorem2}, we recall that
 $$
 \rho(x,y,t)=\xi(x,t)e^{-y}\ \ \ and \ \ \ w(x,z,t)=v(x,y,t)e^{-y}
 $$
where $(\frac{d}{dy})z=e^{-y}$. Forthmore, by the change of variable $z=1-e^{-y}$ in integrals, we obtain the following properties:
$$\|\sqrt{\rho}\|_{L^{\infty}(0,T;L^{2})}=\|\sqrt{\xi}\|_{L^{\infty}(0,T;L^{2})}$$
 $$\|\sqrt{\rho}u\|_{L^{\infty}(0,T;L^{2})}=\|\sqrt{\xi}u\|_{L^{\infty}(0,T;L^{2})}$$
 $$\|\sqrt[3]{\rho}u\|_{L^{3}(0,T;L^{3})}=\|\sqrt[3]{\xi}u\|_{L^{3}(0,T;L^{3})}$$
$$\|\sqrt{\rho}D_{x}(u)\|_{L^{2}(0,T;L^{2})}=\|\sqrt{\xi}D_{x}(u)\|_{L^{2}(0,T;L^{2})}$$
$$\|\sqrt{\nu_{2}}\partial_{y}u\|_{L^{2}(0,T;L^{2})}=\sqrt{\nu_{2}}\|\sqrt{\xi}\partial_{z}u\|_{L^{2}(0,T;L^{2})}$$
$$\|\partial_{y}\rho\|_{L^{\infty}(0,T;L^{2})}=\sqrt{\alpha}\|\xi\|_{L^{\infty}(0,T;L^{2})}$$
where
$$\alpha=\frac{\int_{0}^{1-e^{-h}}(1-z)dz}{1-e^{-h}}$$

Moreover,
\begin{equation*}
\begin{aligned}
\|\sqrt{\rho}\partial_{y}v\|_{L^{2}(0,T;L^{2})}^{2}&=\int_{0}^{T}\int_{0}^{H}\int_{\acute{\Omega_{x}}}|\sqrt{\rho}\partial_{y}v|^{2}dxdzdt\\
&\leq2\int_{0}^{T}\int_{0}^{h}\int_{\acute{\Omega_{x}}}(\xi\partial_{z}w^{2}+\xi w^{2}\frac{1}{(1-z)^{2}})dxdzdt\\
&\leq C\int_{0}^{T}\int_{0}^{h}\int_{\Omega_{x}}\xi\partial_{z}w^{2}dxdzdt
\end{aligned}
\end{equation*}
and
\begin{equation*}
\begin{aligned}
\|\sqrt{\rho}v\|_{L^{2}(0,T;L^{2})}^{2}&=\int_{0}^{T}\int_{0}^{H}\int_{\acute{\Omega_{x}}}|\sqrt{\rho}v|^{2}dxdzdt\\
&=\int_{0}^{T}\int_{0}^{h}\int_{\Omega_{x}}\xi w^{2}\frac{1}{(1-z)^{2}}dxdzdt\\
&\leq C\int_{0}^{T}\int_{0}^{h}\int_{\Omega_{x}}\xi w^{2}dxdzdt
\end{aligned}
\end{equation*}
With the above estimates in hand, Theorem\ref{theorem2} is proved.

{\bf Acknowledgement.}
Dou is partially supported by NSFC (Grant Nos. 11401395, 11471334) and BNSF (Grant No. 1154007), Jiu is supported by National Natural Sciences Foundation of China (No. 11671273 and No. 11231006).


\vspace{2mm}


\end{document}